\documentclass[12pt]{amsart}

\usepackage{verbatim, amssymb, enumitem, mathtools,color}

\usepackage[breaklinks=true,colorlinks=true,linkcolor=blue,citecolor=red,urlcolor=blue,psdextra,pdfencoding=auto]{hyperref}

\renewcommand \a{\alpha}
\renewcommand \b{\beta}
\newcommand \K{\delta}

\newcommand \la{\lambda}

\newcommand \br{\mathbb{R}}

\newcommand \rk{\operatorname{rk}}
\newcommand \Ker{\operatorname{Ker}}

\newcommand \codim{\operatorname{codim}}

\newcommand \Span{\operatorname{Span}}

\newcommand \cA{\mathcal{A}}

\newcommand \mU{\mathcal{U}}

\newcommand \Vg{\mathfrak{V}}

\newcommand \og{\mathfrak{o}}
\newcommand \eg{\mathfrak{e}}

\newcommand\ag{\mathfrak a}
\newcommand\kg{\mathfrak k}
\newcommand\g{\mathfrak g}
\newcommand\q{\mathfrak q}
\newcommand\cs{\mathfrak c}
\newcommand\csp{\cs^\perp}
\newcommand\h{\mathfrak h}
\newcommand\z{\mathfrak z}
\newcommand\m{\mathfrak m}
\newcommand \so{\mathfrak{so}}

\newcommand \gl{\mathfrak{gl}}
\newcommand \ug{\mathfrak{u}}

\newcommand \s{\mathfrak{s}}
\newcommand \vg{\mathfrak{v}}
\newcommand \n{\mathfrak{n}}
\newcommand \f{\mathfrak{f}}

\renewcommand\t{\mathfrak t}

\newcommand \ad{\operatorname{ad}}
\newcommand \Ad{\operatorname{Ad}}

\newcommand \Id{\operatorname{Id}}

\newcommand{\hodge}{{\star}}

\newcommand \<{\langle}
\renewcommand \>{\rangle}
\newcommand \ip{\<\cdot,\cdot\>}

\newtheorem{theorem}{Theorem}
\newtheorem*{theorem*}{Theorem}

\newtheorem*{corollary*}{Corollary}
\newtheorem*{conj*}{Conjecture}
\newtheorem{lemma}{Lemma}
\newtheorem{proposition}{Proposition}
\newtheorem*{prop*}{Proposition}
\newtheorem*{GeoL}{Geodesic Lemma \cite{DK}}

\theoremstyle{definition}

\newtheorem*{definition*}{Definition}

\theoremstyle{remark}
\newtheorem{remark}{Remark}

\newtheorem*{notation*}{Notation}
\newtheorem*{algorithm*}{Algorithm}
\newtheorem*{example*}{Example}

\makeatletter
\@namedef{subjclassname@2020}{%
  \textup{2020} Mathematics Subject Classification}
\makeatother

\begin{document}

\title[Pseudo-Riemannian geodesic orbit nilmanifolds of signature $(n-2,2)$]{Pseudo-Riemannian geodesic orbit nilmanifolds of signature $\boldsymbol{(n-2,2)}$}

\author{Zhiqi Chen}
\address{ZC: School of Mathematics and Statistics, 
	Guangdong University of Technology, 
	Guangzhou 510520, P.R. China}
\email{chenzhiqi@nankai.edu.cn}
\thanks{ZC was partially supported by NNSF of China (11931009 and 12131012) and Guangdong Basic and Applied Basic Research Foundation (2023A1515010001).}

\author{Yuri Nikolayevsky}
\address{YN: Department of Mathematical and Physical Sciences, 
	La Trobe University, VIC 3086, Australia} 
\email{Y.Nikolayevsky@latrobe.edu.au}
\thanks{YN was partially supported by ARC Discovery Grant DP210100951.}

\author{Joseph A. Wolf}
\address{JW: Department of Mathematics, University of California, 
	Berkeley, CA 94720-3840, USA}
\email{jawolf@math.berkeley.edu, josephwolf2@berkeley.edu }
\thanks{JW was partially supported by a Simons Foundation grant.}

\author{Shaoxiang Zhang}
\address{SZ: College of Mathematics and Systems Science, 
	Shandong University of Science and Technology, 
	Qingdao 266590, P.R. China}
\email{zhangshaoxiang@mail.nankai.edu.cn}
\thanks{SZ was partially supported by the National Natural Science Foundation 
	of China (No. 12201358), 
	Natural Science Foundation of Shandong Province (No. ZR2021QA051)}



\begin{abstract} 
The geodesic orbit property is useful and interesting in itself, and it plays 
a key role in Riemannian geometry. It implies homogeneity and has important 
classes of Riemannian manifolds as special cases.  Those classes include weakly 
symmetric Riemannian manifolds and naturally reductive Riemannian manifolds.  
The corresponding results for indefinite metric manifolds are much more 
delicate than in Riemannian signature, but in the last few years important 
corresponding structural results were proved for geodesic orbit Lorentz 
manifolds.  Here we extend Riemannian and Lorentz results to trans-Lorentz 
nilmanifolds. Those are the geodesic orbit pseudo Riemannian manifolds 
$M = G/H$ of signature $(n-2,2)$ such that a nilpotent analytic subgroup of 
$G$ is transitive on $M$.  For that we suppose that there is a reductive 
decomposition $\g = \h \oplus \n \text{ (vector space direct sum) with } 
[\h,\n] \subset \n$ and $\n$ nilpotent. When the metric is nondegenerate 
on $[\n,\n]$ we show that $\n$ is abelian or 2-step nilpotent. That is the 
same result as for geodesic orbit Riemannian and Lorentz nilmanifolds.  When 
the metric is degenerate on $[\n,\n]$ we show that $\n$ is a double extension 
of a geodesic orbit nilmanifold of either Riemannian or Lorentz signature. 
\end{abstract}

\maketitle

\section{Introduction and Statement of Results}
\label{s:intro}

A pseudo-Riemannian manifold $(M,ds^2)$ is called a \emph{geodesic orbit manifold} (or a manifold with homogeneous geodesics, or simply a $GO$ manifold), if every geodesic of $M$ is an orbit of a $1$-parameter subgroup of the full isometry group $I(M) = I(M,ds^2)$. One loses no generality
if one replaces $I(M)$ by its identity component $I^0(M)$.  If $G$ is a transitive Lie subgroup of $I^0(M)$, so $(M,ds^2) = (G/H,ds^2)$ where $H$ is an isotropy subgroup of $G$, and if every geodesic of $M$ is an orbit of a $1$-parameter subgroup of $G$, then we say that $(M,ds^2)$ is a \emph{$G$-geodesic orbit manifold}, or a $G$-$GO$ manifold. Clearly every $G$-$GO$ manifold is a $GO$ manifold, but not vice versa. The class of geodesic orbit manifolds includes (but is not limited to) symmetric spaces, weakly symmetric spaces, normal and generalized normal homogeneous spaces, and naturally reductive spaces. For the current state of knowledge in the theory of Riemannian geodesic orbit manifolds we refer the reader to \cite{BN} and its bibliography.

In this paper, we study the $GO$ condition for pseudo-Riemannian nilmanifolds $(N,ds^2)$, relative to subgroups $G \subset I(N)$ of the form $G = N \rtimes H$, where $H$ is an isotropy subgroup. Most of our results apply to the case where $(N,ds^2)$ is a \emph{trans-Lorentz manifold}, that is, the signature of $ds^2$ is $(n-2,2)$, where $n = \dim N$.

Our results for $G$-$GO$ manifolds $(M,ds^2) = (G/H,ds^2)$ require the coset space $G/H$ to be reductive. In other words, they make use of an $\Ad_G(H)$-invariant decomposition $\g = \m \oplus \h$. Very few structural results are known for indefinite metric $GO$ manifolds that are not reductive, and we always assume that $G/H$ is reductive (see the discussion below).

The $GO$ condition for reductive spaces is well known: 
\begin{GeoL} \label{GeoL}
Let $(M,ds^2)=G/H$ be a reductive pseudo-Riemannian homogeneous space, with the corresponding reductive decomposition $\g = \h \oplus \m$. Then $M$ is a $G$-geodesic orbit space if and only if, for any $T \in \m$, there exist $A = A(T) \in \h$ and $k=k(T) \in \br$ such that if $T' \in \m$ then
\begin{equation}\label{eq:golemma}
  \<[T+A,T']_\m,T\> = k \<T,T'\>,
\end{equation}
where $\ip$ denotes the inner product on $\m$ defined by $ds^2$, and the subscript ${}_\m$ in~\eqref{eq:golemma} means taking the $\m$-component in $\g = \h \oplus \m$.
\end{GeoL}
Note that $k(T)=0$ unless $T$ is a null vector (substitute $T'=T$ in~\eqref{eq:golemma}).

Recall that a pseudo-Riemannian \emph{nilmanifold} is a pseudo-Riemannian manifold admitting a transitive nilpotent Lie group of isometries. In the Riemannian case, the full isometry group of a nilmanifold $(N,ds^2)$, where $N$ is a transitive nilpotent group of isometries, is the semidirect product $I(N) = N \rtimes H$, where $H$ is the group of all isometric automorphisms of $(N,ds^2)$ \cite[Theorem~4.2]{W1}. In other words, $N$ is the nilradical of $I(N)$. In the pseudo-Riemannian cases, $I(N)$ might still contain $N \rtimes H$ and yet be strictly larger. In indefinite signatures of metric a nilmanifold is not necessarily reductive as a coset space of $I(N)$, and even when it is, $N$ does not have to be a normal subgroup of $I(N)$.  Here the $GO$ condition does not rescue us, for there exist $4$-dimensional, Lorentz $GO$ nilmanifolds that are reductive relative to $I(N)$, but for which $N$ is not an ideal in $I(N)$ \cite[Section~3]{dBO}. Moreover, already in dimension $4$ (the lowest dimension for homogeneous pseudo-Riemannian spaces $G/H$ with
$H$ connected that are not reductive), every non-reductive space is a $GO$ manifold when we make a correct choice of parameters~\cite[Theorem~4.1]{CFZ}. These results explain (and motivate) our study of $G$-$GO$ nilmanifolds $G/H = (N \rtimes H)/H$, where $N$ is nilpotent, and $N$ is \emph{the maximal} connected subgroup of isometric automorphisms of $(N,ds^2)$ (although our results remain valid for a smaller subgroup $H$). Given a reductive $G$-$GO$ trans-Lorentz nilmanifolds $(G/H,ds^2)$, where $G = N \rtimes H$, with $N$ nilpotent, and the corresponding reductive decomposition $\g = \h \oplus \n$ at the level of Lie algebras, we denote $\ip$ the inner product on $\n$ induced by $ds^2$, and by $\ip'$, the restriction of $\ip$ to the derived algebra $\n'=[\n,\n]$.

The structure of the paper is as follows.  Section~\ref{s:nondeg} contains
the proof of the first main theorem:

\begin{theorem}\label{th:nondeg}
     Let $(M = G/H, ds^2)$ be a connected trans-Lorentz $G$-geodesic orbit nilmanifold where $G = N \rtimes H$, with $N$ nilpotent. Let $\ip$ denote the inner product on $\n$ induced by $ds^2$. If $\ip|_{\n'}$ is nondegenerate, then $N$ is either abelian or $2$-step nilpotent.
\end{theorem}

\begin{remark} \label{rem:manytrans}
  There are very many connected trans-Lorentz $G$-geodesic orbit nilmanifolds as in Theorem \ref{th:nondeg}.  They are real forms of the complexifications of Riemannian $GO$ spaces.  See \cite[Proposition 4.3 and Corollary 5.4]{CW}
for the collection and \cite{W2} for the fact that those real forms are $GO$.
\hfill $\diamondsuit$ \end{remark}

Theorem \ref{th:nondeg} extends the results of~\cite[Theorem~2.2]{Gor} (for the Riemannian signature) and of~\cite[Theorem~2]{NW} and \cite[Theorem~7]{CWZ} (for the Lorentz signature) to the trans-Lorentz case. 
Our proof of Theorem \ref{th:nondeg} is split into two parts, given in Subsections~\ref{ss:dl} and~\ref{ss:tld}, depending on the signature of $\ip'$.  
In Subsection~\ref{ss:ex} we give an example which shows that the results of Theorem~\ref{th:nondeg} and \cite[Theorem~2]{NW} are ``almost" tight in the sense of the signature: there is a $G$-$GO$ nilmanifold of signature $(8,4)$, with the Lorentz derived algebra, which is $4$-step nilpotent.

In Section~\ref{s:double} we extend the result of~\cite[Theorem~3]{NW} (for 
the Lorentz signature) to the trans-Lorentz (signature $(n-2,2)$) setting.  

In Theorem~\ref{th:deg}, stated just below,  we prove that if the restriction 
$\ip'$ is \emph{degenerate} then $\n$ can be obtained by the 
\emph{double extension} procedure from a metric Lie algebra of either a 
Riemannian or Lorentz nilmanifold. The double extension construction
(which is explained in Section~\ref{s:double}) is a useful tool in 
pseudo-Riemannian homogeneous geometry, in particular in the theory of 
bi-invariant metrics (see the recent survey \cite{Ova}) and in the context 
of $GO$ nilmanifolds \cite[Section~4]{NW}.  The precise result is

\begin{theorem} \label{th:deg} 
Let $(M = G/H, ds^2)$ be a connected trans-Lorentz $G$-geodesic orbit nilmanifold where $G =N \rtimes H$, with $N$ nilpotent. Let $\ip$ denote the inner product on $\n$ induced by $ds^2$. If $\ip|_{\n'}$ is degenerate, then $(\n, \ip)$ is a either a $2$-dimensional double extension of a metric Lie algebra corresponding to a Lorentz $GO$ nilmanifold, or a $4$-dimensional double extension of a metric Lie algebra corresponding to a Riemannian $GO$ nilmanifold.
\end{theorem}

\begin{remark} \label{rem:repeatde} 
  The Lorentz $GO$ nilmanifold in Theorem~\ref{th:deg} is geodesic orbit relative to the group $G_0$, which is the semidirect product of the group of parallel translations and pseudo orthogonal automorphisms, as constructed in Lemma~\ref{l:de}, and hence by the results of~\cite{NW} is either at most $2$-step nilpotent, or is by itself obtained from a $2$-dimensional double extension of a metric Lie algebra of a Riemannian nilmanifold (which must be at most 2-step nilpotent by~\cite[Theorem~2.2]{Gor}). As the composition of two repeated $2$-dimensional double extensions is equivalent to a single $4$-dimensional double extension, we deduce that in the assumptions of Theorem~\ref{th:deg}, the Lie algebra of the nilmanifold $M$ is obtained either by a $2$-dimensional double extension of a Lorentz Lie algebra $\m_0$ or by a $4$-dimensional double extension of a Riemannian Lie algebra $\m_0$, where in both cases, $\m_0$ is at most $2$-step nilpotent.

Note that even a $2$-dimensional $GO$ double extension of an abelian definite Lie algebra can be of an arbitrarily high step, as shown in~\cite[Section~5]{NW}.
\hfill $\diamondsuit$ \end{remark}

The authors have no competing interests to declare that are relevant to the content of this article.

\section{Proof of Theorem 1: If  \texorpdfstring{$ds^2|_{[\n,\n]}$}{ds\unichar{"00B2}|[\unichar{"1D52B},\unichar{"1D52B}]} is nondegenerate then \\ \texorpdfstring{$\n$}{\unichar{"1D52B}} is either abelian or $2$-step nilpotent}
\label{s:nondeg}

Given a reductive homogeneous pseudo-Riemannian manifold $(G/H,ds^2)$, where $G = N \rtimes H$, with $N$ nilpotent, we identify $\n = {\rm Lie}(N)$ with the tangent space to $G/H$ at $1N$. Let $\ip$ be the inner product on $\n$ induces by $ds^2$, and denote $\n'=[\n,\n]$.

Assume that the restriction $\ip'$ of the inner product $\ip$ to $\n'$ is nondegenerate. Denote $\vg=(\n')^\perp$; note that $\n$ is the direct orthogonal sum of $\n'$ and $\vg$, and both subspaces $\n'$ and $\vg$ of $\n$ are $\ad_\g(\h)$-invariant.

\begin{remark} \label{rem:adhinv}
Note that if $V_1$ and $V_2$ are $\ad_\g(\h)$-invariant subspaces of $\n$, then each of the following subspaces is also $\ad_\g(\h)$-invariant:
  \begin{equation*}
    V_1^\perp, \quad V_1 + V_2, \quad V_1 \cap V_2, \quad [V_1, V_2], \quad \{X \in \n \, : \, [X, V_1] \subset V_2\}.
  \end{equation*}
  In particular, the centraliser and the normaliser of an $\ad_\g(\h)$-invariant subspace of $\n$ are themselves $\ad_\g(\h)$-invariant.
\hfill $\diamondsuit$ \end{remark}

Let $(G/H,ds^2)$ be $G$-geodesic orbit. 
Following the first steps in the proof of \cite[Theorem~7]{CWZ} and of \cite[Theorem~1]{NW}, we take $T=X+Y$ and $T' = X'+Y'$, where $X,X' \in \n'$, $Y,Y' \in \vg$, and $T$ is non-null in \eqref{eq:golemma}. Then $k(T)=0$ and there exists $A=A(X,Y) \in \h$ such that
  \begin{equation}\label{eq:gonondeg1}
  \<[A,X'],X\> + \<[A,Y'],Y\> + \<[X,X']+[Y,X']+[X,Y']+[Y,Y'],X\>=0.
  \end{equation}
Taking $Y'=Y, X'=0$ we obtain, by continuity,
  \begin{equation}\label{eq:YXX1}
    \<[Y,X],X\>=0, \quad\text{for all } Y \in \vg, \, X \in \n'.
  \end{equation}
As $\vg$ generates $\n$ it follows that
  \begin{equation}\label{eq:TXX1}
    \<[T,X],X\>=0, \quad\text{for all } T \in \n, \, X \in \n'.
  \end{equation}

\begin{remark} \label{rem:n'def} 
  Note that if $\ip'$ is definite, equation~\eqref{eq:TXX1} implies $[\n,\n']=0$, and so $\n$ is at most $2$-step nilpotent, regardless of the signature of $\ip$. In the context of Theorem~\ref{th:nondeg} we can therefore assume that $\ip'$ is indefinite.
\hfill $\diamondsuit$ \end{remark}

Separating the $X'$- and the $Y'$-components in \eqref{eq:gonondeg1} and using \eqref{eq:YXX1} and \eqref{eq:TXX1}, we find that for all $X \in \n'$ and $Y \in \vg$ with $X+Y$ non-null, there exists $A=A(X,Y) \in \h$ such that for all $X' \in \n', \, Y' \in \vg$,
  \begin{gather}\label{eq:AYY1}
    \<[A,Y],Y'\> = \<[Y,Y'],X\>,\\
    [A+Y,X] =0. \label{eq:AYX1}
  \end{gather}

Denote $\s := \so(\n', \ip') \subset \gl(\n')$, the algebra of skew-symmetric endomorphisms relative to the restriction of $\ip$ to $\n'$. By \eqref{eq:TXX1} $\kg:= \ad_\g(\n)|_{\n'}$ is a subalgebra of $\s$ consisting of nilpotent endomorphisms. In fact, the map $\phi: \n \to \kg$ defined by $\phi(T) = \ad(T)|_{\n'}$ for $T \in \n$ is a Lie algebra homomorphism. Using Engel's Theorem, $\kg$ is triangular. Thus it is conjugate by an inner automorphism \cite[Theorem~2.1]{Mos} to a subalgebra of the nilpotent part $\ug$ of an Iwasawa decomposition $\s=\t \oplus \ag \oplus \ug$. In the following we may (and do) assume $\kg \subset \ug$.

In view of Remark~\ref{rem:n'def}, to prove Theorem~\ref{th:nondeg} we need to consider two cases: when the restrictions of $\ip$ to both $\n'$ and $\vg$ are Lorentz, and when restriction of $\ip$ to $\n$ is trans Lorentz and the restriction to $\vg$ is definite.

We consider these two cases separately in the following two subsections. The proof of Theorem~\ref{th:nondeg} will follow from Propositions~\ref{p:dernondeg2} and~\ref{p:dernondeg1} below.

\subsection{Both \texorpdfstring{$\n'$ and $\vg$}{\unichar{"1D52B}' and \unichar{"1D533}} are Lorentz}
\label{ss:dl}

In this subsection we additionally assume, in the assumptions of Theorem~\ref{th:nondeg}, that the restrictions of $\ip$ to both $\n'$ and $\vg$ are both Lorentz. We prove the following.

\begin{proposition}\label{p:dernondeg2}
     Let $(M = G/H, ds^2)$ be a connected pseudo-Riemannian $G$-geodesic orbit nilmanifold where $G = N \rtimes H$ with $N$ nilpotent. If the restrictions of $\ip$ to both $\n'$ and $\vg$ are of Lorentz signature, then $N$ is either abelian or $2$-step nilpotent.
\end{proposition}

\begin{proof}
Denote $m = \dim \n'$ (note that $m \ge 2$). We adopt the notation and will use the facts stated at the start of this section.

The subalgebra $\s=\so(\n',\ip') \subset \gl(\n')$ of skew-symmetric endomorphisms of $\ip'$ is isomorphic to $\so(m -1, 1)$. We can choose a basis $\{e_1, \dots, e_m\}$ for $\n'$ relative to which the restriction of $\ip$ to $\n'$ and the nilpotent part $\ug$ of the Iwasawa decomposition of $\s$ are given by the following matrices:
  \begin{equation} \label{eq:ip'ug}
    \ip|_{\n'} = \left ( \begin{smallmatrix} 0 & 0 & 1 \\ 0 & I_{m-2} & 0 \\  1 & 0 & 0 \end{smallmatrix} \right ) \text{ and }
	\ug = \left\{ \left ( \begin{smallmatrix} 0 & 0 & 0 \\ u & 0_{m-2} & 0 \\
	0 & -u^t & 0 \end{smallmatrix}\right )  \, : \, u \in \br^{m-2} \right\}.
  \end{equation}

As $\kg \subset \ug$, we obtain a linear map $\Phi: \n \to \Span(e_2, \dots, e_{m-1})$ such that, for all $T \in \n$,
 \begin{equation}\label{eq:advg}
    [T, e_1] = \Phi T,\; [T, e_m] = 0, \text{ and }
	[T, e_i] = -\<\Phi T,e_i\> e_m \text{ for } 2 \le i \le m-1.
 \end{equation}
As $\ug$ (and hence $\kg$) is abelian, we obtain $[[\vg,\vg],\n']=0$. Since $\vg$ generates $\n$, we obtain $[[\n,\n],\n']=0$, which implies that $\n'$ is
abelian.

Introduce the $2$-forms $\omega_i \in \Lambda^2(\n)$ by
  \begin{equation}\label{eq:omega}
    [T_1,T_2]= \sum\nolimits_{i=1}^{m} \omega_i(T_1,T_2)e_i	\text{  for  } T_1, T_2 \in \n.
  \end{equation}
From~\eqref{eq:ip'ug} and~\eqref{eq:advg} we have
\begin{equation}\label{eq:omega1}
  \omega_1(T_1,T_2)=\<[T_1,T_2],e_m\>, \quad \text{and} \quad \omega_1(\n,\n')=0.
\end{equation}
As $\omega_1$ cannot be zero (since $e_1 \in \n'$) we obtain $\omega_1(\vg,\vg) \ne 0$.

Using \eqref{eq:advg} and \eqref{eq:omega}, the Jacobi identity gives
  \begin{equation}\label{eq:Jac}
    \sigma\Big(\omega_1(T_1,T_2)\Phi T_3 + \sum\nolimits_{i=2}^{m-1} \omega_i(T_1,T_2)\<\Phi T_3,e_i\>e_m\Big)=0,
  \end{equation}
where $\sigma$ denotes the cyclic permutation of $T_1,T_2,T_3 \in \n$.

Consider the following two subspaces of $\n$:
\begin{equation}\label{eq:ci}
  \cs = \Ker \Phi ( = \{T \in \n \, : \, [T, \n'] = 0\}), \qquad  \q = \{T \in \n \, : \omega_1(T,\n) ( = \<[T, \n],e_m\>) = 0\}
\end{equation}
(the fact that $\Ker \Phi$ is the centraliser of $\n'$ in $\n$ follows from \eqref{eq:advg}). By \cite[Theorem~1(b)]{NW} we can (and will) assume that $\cs$ is degenerate. Furthermore, we can assume that $\cs \ne \n$ (equivalently, $\Phi \ne 0$), as otherwise the algebra $\n$ is $2$-step nilpotent by \eqref{eq:advg}.

In these notations and assumptions, we have the following. 
{
\begin{lemma} \label{l:ci}
{\ }

   \begin{enumerate}[label=\emph{(\alph*)},ref=\alph*]
	  \item \label{it:ciinv}
	   $\n' \subset \q \subset \cs$, and both $\q$ and $\cs$ are $\ad_\g(\h)$-invariant ideals of $\n$. Moreover, $[\cs,\cs] \subset \z(\n) \cap \n'$.

	  \item \label{it:codimci}
	   $\codim \q = 2$ and $\codim \cs \in \{1,2\}$ \emph{(}equivalently, $\rk \, \Phi \in \{1,2\}$\emph{)}.

	  \item \label{it:eincs}
	   Let $e \in \cs \cap \cs^\perp$ be a nonzero \emph{(}necessarily null\emph{)} vector.
       Then $e \in \vg$, the line $\br e$ is $\ad_\g(\h)$-invariant and $[e,e^\perp]= 0$. 

      \item \label{it:fe}
       Let $f \in \vg$ be a null vector such that $f \notin e^\perp$ and $\<f,e\>=1$. Then $[\h,[f,e]]=0$.

      \item \label{it:eig}
	   $[f,e] \in \z(\n) \cap \n'$ and $e \in \q$.
   \end{enumerate}
\end{lemma}
\begin{proof}
  \eqref{it:ciinv} The fact that $\n' \subset \q$ follows from \eqref{eq:omega1}. Furthermore, taking $T_1 \in \q$ in \eqref{eq:Jac} we obtain $\omega_1(T_2,T_3) \Phi(T_1)=0$ which implies $\q \subset \cs$. As both $\q$ and $\cs$ contain $\n'$, they are ideals of $\n$. The fact that $\cs$ is $\ad_\g(\h)$-invariant follows from Remark~\ref{rem:adhinv}. Moreover, by Remark~\ref{rem:adhinv}, $\q$ is $\ad_\g(\h)$-invariant provided $\br e_m$ is. To see the latter, we note that $[\vg, \n']$ is $\ad_\g(\h)$-invariant. From \eqref{eq:advg} (and the fact that $\Phi \ne 0$), we have $e_m \in [\vg, \n'] \subset \Span(e_2, \dots, e_m)$. Then $[\vg, \n'] \cap ([\vg, \n'])^\perp = \br e_m$ is $\ad_\g(\h)$-invariant, by Remark~\ref{rem:adhinv}. 
  Finally, the fact that $[\cs,\cs] \subset \z(\n) \cap \n'$ follows from the Jacobi identity, as $[\cs,\n']=0$.

  \eqref{it:codimci} If $\rk \Phi \ge 3$, then for almost all triples $T_1,T_2,T_3 \in \n$, the vectors $\Phi T_1, \Phi T_2, \Phi T_3 \in \Span(e_2, \dots, e_{m-1})$ are linearly independent, and so $\omega_1=0$ by \eqref{eq:Jac}. This is a contradiction, so $\rk \Phi \le 2$. As $\Phi \ne 0$ we obtain $\codim \cs \, (= \rk \Phi) \in \{1,2\}$. Now as $\q \subset \cs$ by~\eqref{it:ciinv} and as $\q$ is the null space of the skew-symmetric form $\omega_1$, the codimension of $\q$ must be a positive even number. Furthermore, from \eqref{it:ciinv} we have $\omega_1(\cs,\cs)=0$. If $\codim \cs = 1$, this implies $\codim \q =2$. If $\rk \Phi =2$, we take $T_1,T_2 \in \n$ in \eqref{eq:Jac} such that the vectors $\Phi T_1, \Phi T_2 \in \Span(e_2, \dots, e_{m-1})$ are linearly independent and take $T_3 \in \cs$. We obtain $\omega_1(T_1,\cs)=\omega_1(T_2,\cs)=0$. As $\Span(T_1,T_2) \oplus \cs = \n$ we get $\omega_1(\n, \cs)=0$, and so $\cs \subset \q$ by \eqref{eq:ci} which implies $\cs=\q$ by~\eqref{it:ciinv}.

  \eqref{it:eincs} As $\cs$ is degenerate, the space $\cs \cap \cs^\perp$ has dimension $1$ and is spanned by a (nonzero) null vector $e$. As $\n' \subset \cs$ by~\eqref{it:ciinv} we obtain $e \in \vg$. The subspace $\br e$ is $\ad_\g(\h)$-invariant as $\cs$ is (and by Remark~\ref{rem:adhinv}). Then~\eqref{eq:AYY1} with $Y=e, \, Y' \in e^\perp \cap \vg$ implies that $\<[e,e^\perp \cap \vg],X\>=0$ for all $X \in \n'$ such that $e+X$ is non-null. This gives $[e,e^\perp \cap \vg] = 0$. As $e \in \cs$ we have $[e, \n'] = 0$, and so $[e,e^\perp] = 0$.

  \eqref{it:fe} Choose $f \in \vg$ to be a null vector such that $f \notin \cs$ and $\<f,e\>=1$ (this choice is not unique). Let $A \in \h$. By~\eqref{it:eincs}, $\br e$ is $\ad_\g(\h)$-invariant, and so $[A,e]=ae$, for some $a \in \br$, and so $\<[A,f],e\>=-a$. As $e^\perp \oplus \br f = \n$, we have $[A,f] + af \in e^\perp$. Then $[[A,f],e] = -a [f,e]$ since $[e,e^\perp] = 0$ by~\eqref{it:eincs}. As $[f,[A,e]]=a[f,e]$, the claim follows.

  \eqref{it:eig} Choose $f$ as in \eqref{it:fe}. Then $e^\perp \oplus \br f = \n$, and so from \eqref{it:eincs} we have $[e, \n] \subset \br [f,e]$. Then from \eqref{it:fe} and \eqref{eq:AYX1}, with $X=[f,e]$, we obtain $[Y,[f,e]]=0$, for all $Y \in \vg$ such that $Y+[f,e]$ is non-null, and hence for all $Y \in \vg$. As $\n'$ is abelian, we obtain $[f,e] \in \z(\n) \cap \n'$. But from~\eqref{eq:advg}, $\z(\n) \cap \n' \subset \Span(e_2, \dots, e_m)$ (as $\Phi \ne 0$), and so $\<[f,e],e_m\>=0$ by~\eqref{eq:ip'ug}. So $\<[\n,e],e_m\>=0$ and the claim follows by \eqref{eq:ci}.
\end{proof}
}

Now choose $e \in \cs \cap \cs^\perp$ as in Lemma~\ref{l:ci}\eqref{it:eincs} and choose $f \in \vg \setminus e^\perp$ as in Lemma~\ref{l:ci}\eqref{it:fe}. By Lemma~\ref{l:ci}\eqref{it:ciinv} we have $\n' \subset \q \subset \cs \subset e^\perp$, and by Lemma~\ref{l:ci}\eqref{it:codimci}, $\codim \q = 2$ (and then either $\cs = \q$ or $\cs=e^\perp$). As $f$ is not contained in $e^\perp$, and hence in $\cs$, we have $\Phi f \ne 0$ by~\eqref{eq:advg}. Without loss of generality (scaling $f$ and $e$ and specifying the orthonormal basis $\{e_2, \dots, e_{m-1}\}$) we can assume that $\Phi f = e_2$, and so by \eqref{eq:advg},
\begin{equation}\label{eq:adf}
 [f,e_1]=e_2, \quad [f,e_2] = - e_m,\quad [f,e_i]=0 \; \text{for } i > 2.
\end{equation}
Note that with this choice of the basis, $\z(\n) \cap \n' \subset \Span(e_3, \dots, e_m)$. 

Moreover, from Lemma~\ref{l:ci}\eqref{it:eig} (and Lemma~\ref{l:ci}\eqref{it:ciinv}), the $2$-dimensional subspace $\q^\perp$ contains $e$ and lies in $\vg$. Thus we have $\q^\perp = \Span(e, Y_0)$ for some $Y_0 \notin \q, \, Y_0 \perp e$, with $\<Y_0,Y_0\> \ne 0$. Note that $e^\perp = \br Y_0 \oplus \q$, and so $\br f \oplus \br Y_0 \oplus \q = \n$. As $\omega_1(\q, \n) = 0$ and $\omega_1 \ne 0$, we must have $\omega_1(f,Y_0) \ne 0$. Denote $\kappa_i=\omega_i(f,Y_0)$, so that $[f,Y_0] = \sum_{i=1}^m \kappa_i e_i$ (by~\eqref{eq:omega}), with $\kappa_1 \ne 0$.

As $Y_0 \perp \q$, from~\eqref{eq:AYY1} with $Y=Y_0$ and $Y' \in \q \cap \vg$ we obtain $\<[Y_0,Y'],X\>=0$, for all $Y' \in \q \cap \vg$ and all $X \in \n'$ such that $Y_0+X$ is non-null, which implies $[Y_0,\q \cap \vg] = 0$. Moreover, as $\q^\perp$ is $\ad_\g(\h)$-invariant, for any $A \in \h$ we have $[A,Y_0] \in \Span(e, Y_0)$, and so $[A,Y_0]= \mu e$, for some $\mu \in \br$, since $Y_0$ is non-null and $Y_0 \perp e$.

Then from~\eqref{eq:AYX1} with $X=[f,Y_0]$ we obtain
\begin{equation} \label{eq:fY0}
[[A,f],Y_0]+[f,[A,Y_0]]+[Y,[f,Y_0]]=0,
\end{equation}
for all $Y \in \vg$ such that $Y+[f,Y_0]$ is non-null (note that here we have a particular $A=A(X,Y) \in \h$). Let $[A,f]= a f + b Y_0 + Y'$, where $Y' \in \q \cap \vg$. Then $[[A,f],Y_0] = a [f,Y_0]$ as $[Y_0,\q \cap \vg] = 0$. Furthermore, $[f,[A,Y_0]] = \mu [f,e]$. By Lemma~\ref{l:ci}\eqref{it:eig} we have $[f,e] \in \z(\n) \cap \n'$, and so $[f,e] \in \Span(e_3, \dots, e_m)$. Take $Y = f + \la e$ in \eqref{eq:fY0}, where $\la \in \br$ is chosen in such a way that $Y+X=f+\la e + [f,Y_0]$ is non-null. As $e \in \cs$ we have $[e,[f,Y_0]]=0$, and so we obtain $a [f,Y_0] + [f,[f,Y_0]] \in \Span(e_3, \dots, e_m)$. Substituting $[f, Y_0]= \sum_{i=1}^m \kappa_i e_i$ and using~\eqref{eq:adf} we get $a(\kappa_1 e_1 + \kappa_2 e_2) + \kappa_1 e_2 = 0$ which implies $\kappa_1 = 0$, a contradiction.
\end{proof}

\subsection{Trans-Lorentz \texorpdfstring{$\n'$, definite $\vg$}{\unichar{"1D52B}', definite \unichar{"1D533}}}
\label{ss:tld}

In this subsection we consider the last remaining case in the proof of Theorem~\ref{th:nondeg}. We prove the following.

\begin{proposition}\label{p:dernondeg1}
     Let $(M = G/H, ds^2)$ be a connected pseudo-Riemannian $G$-geodesic orbit nilmanifold where $G = N \rtimes H$ with $N$ nilpotent. Denote $\n'=[\n,\n]$ and $\vg=(\n')^\perp$. If the restriction of $\ip$ to $\n'$ is trans-Lorentz, and the restriction of $\ip$ to $\vg$ is definite, then $N$ is either abelian or $2$-step nilpotent.
\end{proposition}

\begin{proof}
Denote $m = \dim \n'$ (we can assume that $m \ge 4$, as otherwise the claim follows from \cite[Theorem~1(b)]{NW}). We adopt the notation and will use the facts stated at the start of the section.

The subalgebra $\s=\so(\n',\ip') \subset \gl(\n')$ of skew-symmetric endomorphisms of $\ip'$ is isomorphic to $\so(m-2, 2)$. We can choose a basis $\{e_1, \dots, e_m\}$ for $\n'$ relative to which the restriction of $\ip$ to $\n'$ and the nilpotent part $\ug$ of the Iwasawa decomposition of $\s$ are given by the following matrices:
  \begin{equation} \label{eq:ip'ug1}
    \ip|_{\n'} = \left ( \begin{smallmatrix} 0_2 & 0 & I_2 \\ 0 & I_{m-4} & 0 \\  I_2 & 0 & 0_2 \end{smallmatrix} \right ) \text{ and }
	\ug = \left\{ \left ( \begin{smallmatrix} 0 & 0 & 0 & 0 & 0 \\ \a & 0 & 0 & 0 & 0 \\ u & v & 0_{m-4} & 0 & 0 \\
	0 & \b & -u^t & 0 & -\a \\ -\b & 0 & -v^t & 0 & 0 \end{smallmatrix}\right ) \, : \, u, v \in \br^{m-4}, \, \a, \b \in \br \right\}.
  \end{equation}

The homomorphism $\phi: \n \to \ug$ (given by $\phi(T) = \ad(T)|_{\n'}$ for $T \in \n$) defines linear maps $U, V:\n \to \br^{m-4}=\Span(e_3, \dots, e_{m-2})$ and vectors $a, b \in \n$ such that for $T \in \n$, the corresponding entries of the matrix $\phi(T)$ in the notation of~\eqref{eq:ip'ug1} are given by $u =UT, \, v =VT, \, \a=\<a,T\>$ and $\b=\<b,T\>$.

{
\begin{lemma} \label{l:tlrn'ab}
  We have $\phi(\n')=0$, and so the subalgebra $\n'$ is abelian.
\end{lemma}
\begin{proof}
  Clearly $\phi(\n') \subset [\ug, \ug]$. From~\eqref{eq:ip'ug1}, the subalgebra $[\ug,\ug]$ is the subspace of elements of $\ug$ (as given in~\eqref{eq:ip'ug1}) with $v=0$ and $\a=0$, so that for $X \in \n'$, we have $VX=0$ and $\<a, X\>=0$. For $X' \in \n$ we have $0= [X,X'] + [X',X] = \phi(X)X' + \phi(X')X$ which gives $x_1' UX + x_1 UX' = 0$ and $x_1'\<b, X\> + x_1\<b, X'\> =0$, where $x_1$ and $x_1'$ are the $e_1$-components of the vectors $X$ and $X'$, respectively. This implies $U\n'=0$ and $\<b,\n'\>=0$, so $\phi(\n')= 0$, as required. But $\phi(T)X=[T,X]$ for $T \in \n$ and $X \in \n'$, and the second claim follows.
\end{proof}
}

From Lemma~\ref{l:tlrn'ab} it follows that the subalgebra $\kg=\phi(\n) \subset \ug$ is abelian. It is not hard to see using the root decomposition of $\ug$ relative to the abelian subalgebra $\ag \subset \s$ in the Iwasawa decomposition (or to calculate directly), that the algebra $\ug$ contains three different \emph{maximal} abelian subalgebras given below (in the notation of~\eqref{eq:ip'ug1}): 

\begin{enumerate}[label=(\roman*),ref=\roman*]
    \item \label{it:ug1}
	$\ug_1 = \{Q \in \ug \, : \, v=0\}$.

	\item \label{it:ug2}
	$\ug_2 = \br \left ( \begin{smallmatrix} 0 & 0 & 0 & 0 & 0 & 0 \\ 1 & 0 & 0 & 0 & 0 & 0 \\ u_1 & v_1 & 0 & 0 & 0 & 0 \\0 & 0 & 0 & 0_{m-5} & 0 & 0 \\
	0 & 0 & -u_1 & 0 & 0 & -1 \\ 0 & 0 & -v_1 & 0 & 0 & 0 \end{smallmatrix}\right) \oplus
    \left\{ \left ( \begin{smallmatrix} 0 & 0 & 0 & 0 & 0 & 0 \\ 0 & 0 & 0 & 0 & 0 & 0 \\ 0 & 0 & 0 & 0 & 0 & 0 \\w & 0 & 0 & 0_{m-5} & 0 & 0 \\
	0 & \b & 0 & - w^t & 0 & 0 \\ -\b & 0 & 0 & 0 & 0 & 0 \end{smallmatrix}\right ) \, : \, w \in \br^{m-5}, \, \b \in \br \right \}$, \newline where $v_1 \ne 0$
    (up to specifying a basis in $\br^{m-5}=\Span(e_4, \dots, e_{m-2})$).

	 \item \label{it:ug3}
	 $\ug_3$ is a maximal abelian subalgebra of the Heisenberg algebra \newline $\left\{ \left ( \begin{smallmatrix} 0 & 0 & 0 & 0 & 0 \\ 0 & 0 & 0 & 0 & 0\\ u & v & 0_{m-4} & 0 & 0 \\
	0 & \b & - u^t & 0 & 0 \\ -\b & 0 & -v^t & 0 & 0 \end{smallmatrix}\right ) \, : \, u,v \in \br^{m-4}, \, \b \in \br \right \} \subset \ug$. 
\end{enumerate}

We will consider these three cases separately, but following the same pattern. Let $\cs = \Ker \phi \cap \vg$ be the centraliser of $\n'$ in $\vg$, and $\csp$ be its orthogonal complement in $\vg$. By Remark~\ref{rem:adhinv}, both subspaces $\cs$ and $\cs^\perp$ are $\ad_\g(\h)$-invariant. We will always assume that the subspace $\cs^\perp$ is non-trivial (equivalently, $\phi \ne 0$), for otherwise $\n$ is at most $2$-step nilpotent.

Denote $\pi: \h \to \so(\cs^\perp)$ the restriction of the representation of $\h$ to $\cs^\perp$, so that $\pi(A)Y=[A,Y]$ for $A \in \h$ and $Y \in \csp$.

{
\begin{lemma} \label{l:csperp}
  {\ }

   \begin{enumerate}[label=\emph{(\alph*)},ref=\alph*]
      \item \label{it:LLperp}
        If $L \subset \csp$ is an $\ad_\g(\h)$-invariant subspace, then $[L, L^\perp]=0$.
	
	  \item \label{it:pino1}
        The subspace $\csp$ has no $1$-dimensional $\ad_\g(\h)$-invariant subspaces.

      \item \label{it:pi2dim}
	   If $L \subset \csp$ is a $2$-dimensional $\ad_\g(\h)$-invariant subspace, then $[L,L] \subset \z(\n)$, where $\z(\n)$ is the centre of $\n$.

	  \item \label{it:pihnotab}
	   If the subalgebra $\pi(\h) \subset \so(\csp)$ is abelian, then $[\csp,\csp] \subset \z(\n)$.
   \end{enumerate}
\end{lemma}
\begin{proof}
  Assertion~\eqref{it:LLperp} follows from \eqref{eq:AYY1} if we take $Y \in L$ and $Y' \in L^\perp$.

  For assertion~\eqref{it:pino1}, suppose that for a nonzero $Y \in \csp$, the space $\br \, Y$ is $\ad_\g(\h)$-invariant. Then $[\h, Y]=0$, and so $[Y, \vg]=0$ by \eqref{eq:AYX1}. As $\vg$ generates $\n$, we obtain $[Y, \n]=0$, and in particular, $\phi(Y)=0$ contradicting the fact that $Y \in \csp \setminus \{0\}$.

  For assertion~\eqref{it:pi2dim}, suppose that $L=\Span(Y_1,Y_2) \subset \csp$ is $\ad_\g(\h)$-invariant, with the vectors $Y_1$ and $Y_2$ being orthonormal. Then for any $A \in \h$, we obtain that $AY_1$ is a multiple of $Y_2$, and $AY_2$ is a multiple of $Y_1$. Hence $AX=0$, where $X=[Y_1,Y_2]$, and so $[X,\vg]=0$, by \eqref{eq:AYX1}. As $\vg$ generates $\n$, the subspace $[L,L]=\br X$ lies in the centre of $\n$.

  For assertion~\eqref{it:pihnotab}, suppose that the subalgebra $\pi(\h) \subset \so(\csp)$, is abelian. Then $\csp$ is the direct, orthogonal sum of $\ad_\g(\h)$-invariant subspaces of dimension $1$ or $2$ each. But by assertion~\eqref{it:pino1}, there can be no $1$-dimensional subspaces, and then the claim follows from assertions~\eqref{it:pi2dim} and~\eqref{it:LLperp}.
\end{proof}
}

Introduce the $2$-forms $\omega_i \in \Lambda^2(\vg)$ by
  \begin{equation}\label{eq:omegaa}
    [Y_1,Y_2]= \sum\nolimits_{i=1}^{m} \omega_i(Y_1,Y_2)e_i	\text{  for  } Y_1, Y_2 \in \vg.
  \end{equation}
Then
\begin{equation}\label{eq:three}
[Y_3,[Y_1,Y_2]] = \sum\nolimits_{i=1}^{m} \omega_i(Y_1,Y_2) \phi(Y_3) e_i.
\end{equation}
From Lemma~\ref{l:tlrn'ab} (and the fact that $[\cs,\n']=0$) we have $\n'=[\vg,\vg] + [\vg, \n'] = [\vg,\vg] + \phi(\csp)\n'$.

We now separately consider three cases for $\phi(\n)$ as given above.

\underline{Case~\eqref{it:ug1}}: $\phi(\n) \subset \ug_1$. Then $\phi(\csp)\n' \subset \Span(e_2, \dots,e_m)$. As $e_1 \in \n'= [\vg,\vg] + \phi(\csp)\n'$, we obtain $\omega_1(\vg,\vg) \ne 0$. The Jacobi identity gives $\sigma(\omega_1(Y_1,Y_2)(\<a, Y_3\> e_2 + UY_3 - \<b,Y_3\>)=0$, where $\sigma$ denotes the cyclic permutation of $Y_1,Y_2,Y_3 \in \vg$, which can be written as
\begin{equation} \label{eq:Jac1}
\sigma(\omega_1(Y_1,Y_2)\phi(Y_3))=0.
\end{equation}
Taking $Y_1, Y_2 \in \cs$ and $Y_3 \in \csp$ we obtain $\omega_1(\cs,\cs) = 0$. By Lemma~\ref{l:csperp}\eqref{it:LLperp} we have $[\cs,\csp]=0$, and so by \eqref{eq:omegaa} we obtain that also $\omega_1(\cs,\csp) = 0$. As $\omega_1(\vg,\vg) \ne 0$, we deduce that  $\omega_1(\csp,\csp) \ne 0$. Now if $\rk \phi (=\dim \csp) > 2$, then the elements $\phi(Y_1), \, \phi(Y_2)$ and $\phi(Y_3)$ are linearly independent for almost all triples of vectors $Y_1, Y_2, Y_3 \in \vg$, and so~\eqref{eq:Jac1} implies $\omega_1(\vg,\vg) = 0$, a contradiction. By Lemma~\ref{l:csperp}\eqref{it:pino1}, we have $\dim \csp > 1$, and so the only remaining possibility is $\dim \csp = 2$. But then from Lemma~\ref{l:csperp}\eqref{it:pi2dim} we obtain $[\csp,\csp] \subset \z(\n)$. Taking $Y_1, Y_2, Y_3 \in \csp$ in equation~\eqref{eq:three}, we get $\omega_1(\csp,\csp)=0$, a contradiction.

\underline{Case~\eqref{it:ug2}}: $\phi(\n) \subset \ug_2$. For $Y \in \vg$, we have
\begin{equation}\label{eq:phiY2}
  \phi(Y)=
  \left ( \begin{smallmatrix} 0 & 0 & 0 & 0 & 0 & 0 \\ \<a,Y\> & 0 & 0 & 0 & 0 & 0 \\ \la \<a,Y\> & \mu \<a,Y\>  & 0 & 0 & 0 & 0 \\WY & 0 & 0 & 0_{m-5} & 0 & 0 \\
	0 & \<b,Y\>  & -\la \<a,Y\>  & -(WY)^t & 0 & -\<a,Y\>  \\ -\<b,Y\> & 0 & -\mu \<a,Y\>  & 0 & 0 & 0 \end{smallmatrix}\right),
\end{equation}
where $\la, \mu \in \br$ and $W: \vg \to \br^{m-5}=\Span(e_4, \dots, e_{m-2})$. We can assume that $a \ne 0$ and $\mu \ne 0$, for otherwise $\phi(\n) \subset \ug_1$.

Arguing similarly to the previous case, we see that $\phi(\csp)\n' \subset \Span(e_2, \dots,e_m)$, and so we must have $\omega_1(\vg,\vg) \ne 0$. From the Jacobi identity we obtain
\begin{equation} \label{eq:Jac2}
\sigma(\omega_1(Y_1,Y_2)(\<a,Y_3\> e_2 +WY_3))=0,
\end{equation}
where $\sigma$ denotes the cyclic permutation of $Y_1,Y_2,Y_3 \in \vg$. From~\eqref{eq:Jac2} with $Y_1, Y_2 \in \cs$ and $Y_3 \in \csp$ we obtain $\omega_1(\cs,\cs) = 0$. Furthermore, we have $[\cs,\csp]=0$ by Lemma~\ref{l:csperp}\eqref{it:LLperp}, and so $\omega_1(\cs,\csp) = 0$ by \eqref{eq:omegaa}. It follows that $\omega_1(\cs, \vg)=0$, and so we must have $\omega_1(\csp,\csp) \ne 0$, as $\omega_1(\vg,\vg) \ne 0$.

From the $e_2$-component of~\eqref{eq:Jac2} we obtain $\omega_1 \wedge \a = 0$, where $\a$ is the $1$-form on $\vg$ defined by $\a(Y)=\<a,Y\>$. By generalised Cartan's Lemma \cite[Lemma~1]{Aga} we get $\omega_1 = \gamma \wedge \a$ for some $1$-form $\gamma \in \vg^*$. As $\omega_1 \ne 0$, the $1$-form $\gamma$ is not a multiple of $\alpha$. Moreover, as $\omega_1(\cs, \vg)=0$, both the vector $a$ and the vector $c \in \vg$ dual to $\gamma$ lie in $\csp$. Taking the inner product of~\eqref{eq:Jac2} with $e_s, \, s=4, \dots, m-2$, we find that $W^t e_s \in \Span(a,c)$. Hence in the matrix $\phi(Y)$ given in~\eqref{eq:phiY2}, for all $Y \in (\Span(a,c))^\perp \cap \vg$, we have $WY=\<a,Y\>=0$. We first suppose that $b \in \Span(a,c)$. Then $\csp \subset \Span(a,c)$, and as $\dim \csp > 1$ by Lemma~\ref{l:csperp}\eqref{it:pino1}, we deduce that $\dim \csp = 2$ and hence, that $[\csp,\csp]$ lies in the centre of $\n$, by Lemma~\ref{l:csperp}\eqref{it:pi2dim}. But now if we take $Y_1, Y_2, Y_3 \in \csp$ with $\<a,Y_3\> \ne 0$ in~\eqref{eq:three}, then from the $e_2$-component we get $\omega_1(\csp,\csp)=0$ which is a contradiction.

We therefore suppose that $b \notin \Span(a,c)$, and so $\Span(a,b) \subset \csp \subset \Span(a,b,c)$ from~\eqref{eq:phiY2}. Note that the subspace $[\vg,[\vg,\n']]$ is $\ad_\g(\h)$-invariant by Remark~\ref{rem:adhinv}. As $[Y,\n']=\phi(Y)\n'$ for $Y \in \vg$, equation~\eqref{eq:phiY2} gives that the subspace $[\vg,[\vg,\n']]$ lies in $\Span(e_3, \dots, e_m)$. Moreover, as $\phi(Y)^2e_3=\mu\<a,Y\>e_{m-1}$ and $\phi(Y)^2e_2 = -\mu\<a,Y\> (\la\<a,Y\>e_{m-1} + \mu\<a,Y\>e_m)$, and as $a \ne 0$ and $\mu \ne 0$, we obtain that the subspace $[\vg,[\vg,\n']]$ contains $\Span(e_{m-1}, e_m)$. But then $[\vg,[\vg,\n']] \cap ([\vg,[\vg,\n']])^\perp = \Span(e_{m-1}, e_m)$, and so the subspace $\Span(e_{m-1}, e_m)$ is $\ad_\g(\h)$-invariant by Remark~\ref{rem:adhinv}. Then the subspace $\{Y \in \csp \, : \, [Y,\n'] \subset \Span(e_{m-1}, e_m)\}$ is also $\ad_\g(\h)$-invariant by Remark~\ref{rem:adhinv}. But the latter subspace is given by $\{Y \in \csp \, : \, WY=0, \, \<a,Y\>=0\} = \csp \cap (\Span(a,c))^\perp=\br Y_0$, where $Y_0 \ne 0$ is the component of $b$ orthogonal to $\Span(a,c)$. This gives a $1$-dimensional $\ad_\g(\h)$-invariant subspace of $\csp$, in contradiction with Lemma~\ref{l:csperp}\eqref{it:pino1}.

\underline{Case~\eqref{it:ug3}}: $\phi(\n) \subset \ug_3$. This is the most involved case. For $Y \in \vg$, we have
\begin{equation}\label{eq:phiY3}
  \phi(Y)=
  \left ( \begin{smallmatrix} 0 & 0 & 0 & 0 & 0 \\ 0 & 0 & 0 & 0 & 0 \\ UY & VY & & 0_{m-4} & 0 & 0 \\
	0 & \<b,Y\>  & -(UY)^t & 0 & 0  \\ -\<b,Y\> & 0 & -(VY)^t & 0 & 0 \end{smallmatrix}\right),
\end{equation}
where $b \in \vg$, and $U,V: \vg \to \br^{m-4}=\Span(e_3, \dots, e_{m-2})$ are such that for all $Y_1, Y_2 \in \vg$ we have
\begin{equation}\label{eq:UV}
  \<UY_1,VY_2\>=\<UY_2,VY_1\> \quad \text{(equivalently, the matrix $U^tV$ is symmetric)}.
\end{equation}
The following lemma sorts out the ``non-generic" cases.
{
\begin{lemma} \label{l:case3gen}
  Suppose $\phi(\n) \not\subset \ug_1$ and $\phi(\n) \not\subset \ug_2$ $($up to specifying the basic vectors $e_1,e_2,e_{m-1},e_m$, but keeping the form of $\ip|_{\n'}$ given in~\eqref{eq:ip'ug1}$)$. We have the following.
  \begin{enumerate}[label=\emph{(\alph*)},ref=\alph*]
    \item \label{it:UYVYli}
    For almost all $Y \in \vg$, the vectors $UY$ and $VY$ are linearly independent.

    \item \label{it:m-1madh}
    The subspace $\Span(e_{m-1}, e_m)$ is $\ad_\g(\h)$-invariant.

    \item \label{it:cscsp}
    $\cs=\Ker U \cap \Ker V,\; \csp=(\Ker U \cap \Ker V)^\perp$, and so for almost all $Y \in \csp$, the vectors $UY$ and $VY$ are linearly independent.

    \item \label{it:U'eV'eli}
    For almost all $e \in \Span(e_3, \dots, e_{m-2})$, the $1$-forms $\xi_e$ and $\eta_e$ on $\vg$ defined by $\xi_e(Y)=\<UY,e\>$ and $\eta_e(Y)=\<VY,e\>$ are linearly independent.
  \end{enumerate}
\end{lemma}
\begin{proof}
  We cannot have $U=V=0$, as otherwise the subspace $\csp$ is at most $1$-dimensional, in contradiction with Lemma~\ref{l:csperp}\eqref{it:pino1}. Furthermore, if $U$ and $V$ are proportional, we can specify the vectors $e_1,e_2,e_{m-1},e_m$ (without changing the form of $\ip|_{\n'}$ given in~\eqref{eq:ip'ug1}) in such a way that $V=0$ hence obtaining $\phi(\n) \subset \ug_1$. We can therefore assume that $U$ and $V$ are not proportional.

  We use the following well known fact. If $F_1, F_2: \br^p \to \br^q$ are linear maps such that $\rk(F_1x,F_2x) \le 1$, for all $x \in \br^p$, then either $F_1$ and $F_2$ are proportional, or there exist $w \in \br^q$ and $\la_1, \la_2 \in (\br^p)^*$ such that $F_1x=\la_1(x)w$ and $F_2x=\la_2(x)w$, for all $x \in \br^p$.

  For assertion~\eqref{it:UYVYli}, we apply the above fact to $U$ and $V$. As we assume that $U$ and $V$ are not proportional, we obtain that $UY=\la_1(Y)w,\; VY=\la_2(Y)w$ for non-proportional $1$-forms $\la_1, \la_2 \in \vg^*$ and for some $w \ne 0$. But this leads to a contradiction with \eqref{eq:UV}.

  To prove assertion~\eqref{it:m-1madh} we note that from assertion~\eqref{it:UYVYli} and from~\eqref{eq:phiY3} it follows that $\Span(e_{m-1}, e_m) \subset [\vg,\n'] \subset \Span(e_3, \dots, e_m)$, and so $[\vg,\n'] \cap ([\vg,\n'])^\perp=\Span(e_{m-1}, e_m)$. Hence the subspace $\Span(e_{m-1}, e_m)$ is $\ad_\g(\h)$-invariant by Remark~\ref{rem:adhinv}.

  For assertion~\eqref{it:cscsp} we note that from~\eqref{eq:phiY3} we have $\csp = (\Ker U \cap \Ker V)^\perp + \br b$. But if $b \notin (\Ker U \cap \Ker V)^\perp$, then the subspace $\{Y \in \csp \, : \, [Y,\n'] \subset \Span(e_{m-1}, e_m)\}=\{Y \in \csp \, : \, UY=VY=0\}$ is $1$-dimensional and is $\ad_\g(\h)$-invariant by assertion~\eqref{it:m-1madh} and Remark~\ref{rem:adhinv}, in contradiction with Lemma~\ref{l:csperp}\eqref{it:pino1}. It follows that $b \in (\Ker U \cap \Ker V)^\perp$, and so $\csp = (\Ker U \cap \Ker V)^\perp$ and $\cs = \Ker U \cap \Ker V$, as required. Then assertion~\eqref{it:UYVYli} implies that for almost all $Y \in \csp$, we have $\rk(UY|VY) = 2$.

  For assertion~\eqref{it:U'eV'eli}, we apply the above linear-algebraic fact to the conjugates of $U$ and $V$. As we assume $U$ and $V$ to be not proportional, the condition that the $1$-forms $\xi_e$ and $\eta_e$ on $\vg$ are linearly dependent for all $e \in \Span(e_3, \dots, e_{m-2})$ would imply the existence of $w \in \vg \setminus \{0\}$ and $e,e' \in \Span(e_3, \dots, e_{m-2})$ such that $UY=\<w,Y\>e$ and $VY=\<w,Y\>e'$, for all $Y \in \vg$. But then by assertion~\eqref{it:cscsp}, $\csp = \br w$ which contradicts Lemma~\ref{l:csperp}\eqref{it:pino1}.
\end{proof}
}
As Cases~\eqref{it:ug1} and~\eqref{it:ug2} have been already understood, for the rest of the proof we will assume that the conditions of Lemma~\ref{l:case3gen} are satisfied.

As $\phi(\vg)\n' \subset \Span(e_3, \dots, e_m)$, in order to have both $e_1$ and $e_2$ in $\n'$, we need the $2$-forms $\omega_1, \omega_2 \in \Lambda^2(\vg)$ defined by~\eqref{eq:omegaa} to be linearly independent. From the Jacobi identity we obtain
\begin{equation} \label{eq:Jac3}
\sigma(\omega_1(Y_1,Y_2) UY_3 + \omega_2(Y_1,Y_2) VY_3)=0,
\end{equation}
where $\sigma$ denotes the cyclic permutation of $Y_1,Y_2,Y_3 \in \vg$. Taking $Y_1,Y_2 \in \cs$ and $Y_3 \in \csp$ in such a way that $\rk(UY_3|VY_3) = 2$ we obtain $\omega_1(\cs,\cs)=\omega_2(\cs,\cs)=0$. As we also have $\omega_1(\cs,\csp)=\omega_2(\cs,\csp)=0$ by Lemma~\ref{l:csperp}\eqref{it:LLperp} and~\eqref{eq:omegaa}, we obtain $\omega_1(\cs,\vg)=\omega_2(\cs,\vg)=0$, and hence the restrictions of $\omega_1$ and $\omega_2$ to $\csp$ must be linearly independent.

\begin{lemma} \label{l:csple4}
  In the assumptions of Lemma~\ref{l:case3gen} we have the following.
  \begin{enumerate}[label=\emph{(\alph*)},ref=\alph*]
    \item \label{it:dimcsp}
    $\dim \csp \le 4$.

    \item \label{it:normaliser}
    Introduce $K_1, K_2 \in \so(\csp)$ by $\<K_iY,Y'\>=\omega_i(Y,Y')$ for $Y, Y' \in \csp$ and $i=1,2$. Denote $S=\Span(K_1, K_2) \subset \so(\csp)$ \emph{(}note that $\dim S = 2$\emph{)}. Then the subalgebra $\pi(\h) \subset \so(\csp)$ normalises the subspace $S$.
  \end{enumerate}
\end{lemma}
\begin{proof}
  For assertion~\eqref{it:dimcsp}, consider the pencil $\mu_1 \omega_1 + \mu_2 \omega_2 \subset \Lambda^2(\vg)$, where $\mu_1, \mu_2 \in \br$. Suppose at least one element of this pencil has rank greater than or equal to $4$. Specifying the vectors $e_1,e_2,e_{m-1},e_m$ we can assume, without loss of generality, that $\rk \omega_1 \ge 4$.

  Let $\mU \subset \Span(e_3, \dots, e_{m-2})$ be the subset of those vectors $e$ for which the $1$-forms $\xi_e, \eta_e \in \vg^*$ are linearly independent. By Lemma~\ref{l:case3gen}\eqref{it:U'eV'eli}, the subset $\mU$ is open and dense in $\Span(e_3, \dots, e_{m-2})$. Taking the inner product of~\eqref{eq:Jac3} with $e \in \mU$ we obtain $\omega_1 \wedge \xi_e + \omega_2 \wedge \eta_e=0$. By generalised Cartan's Lemma \cite[Lemma~1]{Aga}, there exist $1$-forms $\gamma_{11}, \gamma_{12}= \gamma_{21}, \gamma_{22} \in \vg^*$ such that $\omega_1 = \gamma_{11} \wedge \xi_e + \gamma_{12} \wedge \eta_e$ and $\omega_2 = \gamma_{21} \wedge \xi_e + \gamma_{22} \wedge \eta_e$. In particular, $\rk \omega_1 \le 4$, and so $\rk \omega_1 = 4$ by our assumption. Let $L_1=\{Y \in \vg \, : i_Y (\omega_1)=0\}$. Then $L_1$ has codimension $4$, and $\xi_e(L_1)= \eta_e(L_1)=0$, for all $e \in \mU$, and hence for all $e \in \vg$. It follows that $UL_1=VL_1=0$, and so by Lemma~\ref{l:case3gen}\eqref{it:cscsp}, $\csp = (\Ker U \cap \Ker V)^\perp \subset L_1^\perp$ which implies $\dim \csp \le 4$.

  Now suppose that $\rk(\mu_1 \omega_1 + \mu_2 \omega_2) < 4$, for all $\mu_1, \mu_2 \in \br$. As the rank is always even and as $\omega_1$ and $\omega_2$ are linearly independent, we obtain $\rk(\mu_1 \omega_1 + \mu_2 \omega_2) = 2$, for all $(\mu_1, \mu_2) \in \br^2\setminus \{(0,0)\}$. Then it is easy to see that there exist three linearly independent $1$-forms $\zeta_1, \zeta_2, \zeta_3 \in \vg^*$ such that $\omega_1=\zeta_1 \wedge \zeta_3$ and $\omega_2=\zeta_2 \wedge \zeta_3$. From~\eqref{eq:Jac3} we get $\omega_1 \wedge \xi_e + \omega_2 \wedge \eta_e=0$, for all $e \in \Span(e_3, \dots, e_{m-2})$ which gives $\zeta_1 \wedge \zeta_3 \wedge \xi_e + \zeta_2 \wedge \zeta_3 \wedge \eta_e=0$. It follows that $\zeta_1 \wedge \zeta_2 \wedge \zeta_3 \wedge \xi_e=0$, and so $\xi_e \in \Span(\zeta_1, \zeta_2, \zeta_3)$, and similarly $\eta_e \in \Span(\zeta_1, \zeta_2, \zeta_3)$, for all $e \in  \Span(e_3, \dots, e_{m-2})$. But then the common kernel $L_2$ of the $1$-forms $\zeta_1, \zeta_2, \zeta_3$ has codimension $3$ and lies in the kernel of both $U$ and $V$. It follows that $\csp = (\Ker U \cap \Ker V)^\perp \subset L_2^\perp$ and so $\dim \csp \le 3$.

  For assertion~\eqref{it:normaliser}, we first note that the subspace $\Span(e_3, \dots,e_m)=(\Span(e_{m-1},e_m))^\perp$ is $\ad_\g(\h)$-invariant by Lemma~\ref{l:case3gen}\eqref{it:m-1madh} and Remark~\ref{rem:adhinv}. Now let $A \in \h$ and $Y,Y' \in \csp$. Then by \eqref{eq:omega}, the component of the vector $[AY,Y']+[Y,AY']$ lying in $\Span(e_1,e_2)$ equals $(\omega_1(AY,Y')+\omega_1(Y,AY'))e_1+(\omega_2(AY,Y')+\omega_2(Y,AY'))e_2=\<[K_1,\pi(A)] Y,Y'\> e_1 + \<[K_2,\pi(A)] Y,Y'\>e_2$. As $\Span(e_3, \dots,e_m)$ is $\ad_\g(\h)$-invariant, the component of the vector $A[Y,Y']$ lying in $\Span(e_1,e_2)$ equals  the component of the vector $A(\omega_1(Y,Y')e_1+\omega_2(Y,Y')e_2)$ lying in $\Span(e_1,e_2)$, which is $\omega_1(Y,Y') (A_{11}e_1+A_{12}e_2) + \omega_2(Y,Y') (A_{21}e_1+A_{22}e_2)= \<(A_{11}K_1+A_{21}K_2)Y,Y'\>e_1 + \<(A_{12}K_1+A_{22}K_2)Y,Y'\>e_2$, where $A_{ij}$ denotes the corresponding entry of the matrix of $(\ad(A))|_{\n'}$ relative to the basis $\{e_1,e_2, \dots, e_m\}$. We deduce that $[K_1,\pi(A)], [K_2,\pi(A)] \in \Span(K_1,K_2)$, as required.
\end{proof}

By Lemma~\ref{l:csperp}\eqref{it:pino1} and Lemma~\ref{l:csple4}\eqref{it:dimcsp}, we have $2 \le \dim\csp \le 3$. Moreover, by Lemma~\ref{l:csple4}\eqref{it:normaliser}, the subalgebra $\pi(\h) \subset \so(\csp)$ normalises the $2$-dimensional subspace $S=\Span(K_1,K_2) \!\subset \so(\csp)$. If $\pi(\h)$ is abelian, then by Lemma~\ref{l:csperp}\eqref{it:pihnotab}, we obtain $[\csp, \csp] \subset \z(\n)$. Then taking $Y_1,Y_2,Y_3 \in \csp$ in~\eqref{eq:three} we get $\sum\nolimits_{i=1}^{m} \omega_i(Y_1,Y_2) \phi(Y_3) e_i=0$ which by~\eqref{eq:phiY3} gives $\omega_1(Y_1,Y_2)UY_3 + \omega_2(Y_1,Y_2)VY_3 =0$. As by Lemma~\ref{l:case3gen}\eqref{it:cscsp}, the vectors $UY_3$ and $VY_3$ are linearly independent for almost all $Y_3 \in \csp$, we deduce that $\omega_1(\csp,\csp)=\omega_2(\csp,\csp)=0$, a contradiction. So the subalgebra $\pi(\h) \subset \so(\csp)$ is non-abelian. It is easy to see that the only possible case when the normaliser of a two-dimensional subspace $S$ of a subalgebra $\so(\csp),\; \dim \csp \in \{2,3,4\}$, is non-abelian is the following: $\dim \csp = 4$, so that $\so(\csp) = \so(4) = \so(3) \oplus \so(3)$ (direct sum of ideals), and then $S$ lies in one of the two $\so(3)$-components.

In this last remaining case, denote $\tilde{\omega}_j \in \Lambda^2(\csp), \; j=1,2$, the restriction of the $2$-form $\omega_j \in \Lambda^2(\vg)$ to $\csp$, and for $e \in \Span(e_3, \dots, e_{m-2})$, denote $\tilde{\xi}_e,\tilde{\eta}_e \in (\csp)^*$ the restrictions of the $1$-forms $\xi_e,\eta_e \in \vg^*$, respectively (so that for $Y \in \csp$ we have $\tilde{\xi}_e(Y)=\<UY,e\>$ and $\tilde{\eta}_e(Y)=\<VY,e\>$). Restricting equation~\eqref{eq:Jac3} to $\csp$ (note that $U\cs=V\cs=0$ and $\omega_1(\cs,\vg)=\omega_2(\cs,\vg)=0$ anyway), we obtain $\tilde{\omega}_1 \wedge \tilde{\xi}_e + \tilde{\omega}_2 \wedge \tilde{\eta}_e=0$, for all $e \in \Span(e_3, \dots, e_{m-2})$. Applying the Hodge star operator (and noting that the dual vectors to $\tilde{\xi}_e$ and $\tilde{\eta}_e$ are $U^te, V^te \in \csp$, respectively) gives $i_{U^te}(\hodge\tilde{\omega}_1)+i_{V^te}(\hodge\tilde{\omega}_2)=0$. Let $\hodge K_j \in \so(\csp), \; j=1,2$, be defined by $\<(\hodge K_j) Y,Y'\>=\hodge\tilde{\omega}_j(Y,Y')$ for $Y,Y' \in \csp$. Then from the latter equation we obtain $\<(\hodge K_1)U^te,Y\>+\<(\hodge K_2)V^te,Y\>=0$, for all $Y \in \csp$ and all $e \in \Span(e_3, \dots, e_{m-2})$. This is equivalent to
\begin{equation}\label{eq:stars}
\tilde{U} (\hodge K_1) + \tilde{V} (\hodge K_2) = 0,
\end{equation}
where $\tilde{U}$ and $\tilde{V}$ are the restrictions of $U$ and $V$ to $\csp$, respectively. Note that $K_1$ and $K_2$ are linearly independent and belong to the same $\so(3)$-component of the algebra $\so(\csp)=\so(3) \oplus \so(3)$ (direct sum of ideals). As these components are $\hodge$-invariant, we obtain that $\hodge K_1$ and $\hodge K_2$ are also linearly independent and belong to the same $\so(3)$-component. This implies that $(\hodge K_2)^2=\mu \Id$ for some $\mu < 0$ and that $(\hodge K_1) (\hodge K_2) = \nu \Id + K_3$, where $\nu \in \br$, and $K_3 \ne 0$ belongs to the same $\so(3)$-component of $\so(\csp)$ as $\hodge K_1$ and $\hodge K_2$. In particular, $\det K_3 \ne 0$. We now multiply~\eqref{eq:stars} by $\tilde{U}^t$ on the left and by $\hodge K_2$ on the right. We get $\tilde{U}^t\tilde{U}(\nu \Id + K_3)+\mu \tilde{U}^t\tilde{V}=0$. But $U^tV$ is symmetric by~\eqref{eq:UV}, and so $\tilde{U}^t\tilde{V}$ is symmetric (as $U\cs=V\cs=0$) which implies that the $4 \times 4$ matrix $\tilde{U}^t\tilde{U} K_3$ is also symmetric. Choosing a basis for $\csp$ which diagonalises the semi-definite matrix $\tilde{U}^t\tilde{U}$ we find that $\tilde{U}^t\tilde{U} K_3$ can be symmetric only when $\tilde{U}^t\tilde{U} K_3=0$. But as $\det K_3 \ne 0$, this implies $\tilde{U}=0$, that is, $U\csp=0$, which in combination with $U\cs=0$ gives $U=0$. This is a contradiction with Lemma~\ref{l:case3gen}\eqref{it:UYVYli} which completes the proof of the proposition and of Theorem~\ref{th:nondeg}.
\end{proof}

\subsection{Example}
\label{ss:ex}

The following example shows that Theorem~\ref{th:nondeg}, concerning the 
$2$-step property of $\n$ for the case when the derived algebra $\n'$ is nondegenerate, is ``almost" tight in terms of the signature. We construct a nilpotent, metric Lie algebra $(\n,\ip)$ with the following properties:

\begin{itemize}
  \item $\dim \n = 12,\; \dim \n'=4, \, \dim \vg = 8$.
  \item $\n'$ is Lorentz and $\vg$ is of signature $(5,3)$, so that $\n$ is of signature $(8,4)$.
  \item $\n$ is 4-step nilpotent (and $\n'$ is abelian).
  \item $(\n,\ip)$ is $G$-geodesic orbit (for $G$ as in Theorem~\ref{th:nondeg}). 
\end{itemize}

We define $\n = \vg \oplus \n'$, where $\dim \n'=4,\; \dim \vg = 8$. We have a basis $\{e_1, e_2,e_3,e_4\}$ for $\n'$, and a basis $\{f_1, \dots, f_8\}$ for $\vg$. The inner product $\ip$ is defined in such a way that $\vg \perp \n'$, and

  \begin{equation*} 
	\ip|_{\vg} = \left ( \begin{matrix} 0 & 0 & 0 & 0 & I_2 \\ 0 & 0 & 0 & 1 & 0 \\ 0 & 0 & I_2 & 0 & 0 \\ 0 & 1 & 0 & 0 & 0 \\ I_2 & 0 & 0 & 0 & 0 \end{matrix} \right )\, \text{ and }
    \ip|_{\n'} = \left ( \begin{matrix} 0 & 0 & 1 \\ 0 & I_2 & 0 \\  1 & 0 & 0 \end{matrix} \right ). 
  \end{equation*}

The Lie bracket is defined as follows:
\begin{equation*} 
\begin{gathered}
  [f_1, e_1]=e_2,\quad [f_2, e_1]=e_3,\quad [f_1, e_2]=[f_2, e_3]=-e_4,\\
  [f_1,f_2]=e_1,\quad [f_1,f_6]=e_2,\quad [f_2,f_6]=e_3,\quad  [f_1, f_4]=[f_2, f_5]=e_4.
\end{gathered}
\end{equation*}
It is not hard to see that the algebra $\n'$ so defined is $4$-step nilpotent (in fact, if we disregard the inner product, our algebra $\n'$ is the direct sum of the $6$-dimensional ideal $\Span(f_1,f_2,e_1,e_2,e_3,e_4)$ (the algebra $L_{6,21}(1)$ in \cite{dG}) and the $6$-dimensional abelian ideal $\Span(f_3, f_4+e_2, f_5+e_3, f_6-e_1, f_7, f_8)$).

We now define, for every $T=X+Y$, where $X=\sum_{i=1}^{4} x_i e_i \in \n', \; Y=\sum_{j=1}^8 y_j f_j \in \vg$, the linear operator $\cA$ on $\n$ such that $\cA\n' \subset \n',\, \cA\vg \subset \vg$, and relative to the chosen bases for $\vg$ and $\n'$,
\begin{gather*}
  \cA|_{\vg}=\left ( \begin{smallmatrix}
  0 & 0 & 0 & 0 & 0 & 0 & 0 & 0 \\
  0 & 0 & 0 & 0 & 0 & 0 & 0 & 0\\
  x_2+y_4 & x_3+y_5 & 0 & -y_1 & -y_2 & 0 & 0 & 0 \\
  x_1-y_6 & 0 & 0 & 0 & 0 & y_1 & 0 & 0 \\
  0 & x_1-y_6 & 0 & 0 & 0 & y_2 & 0 & 0 \\
  y_2 & -y_1 & 0 & 0 & 0 & 0 & 0 & 0 \\
  0 & y_3-x_4 & -y_2 & y_6-x_1 & 0 & -x_2-y_4 & 0 & 0 \\
  x_4-y_3 & 0 & y_1 & 0 & y_6-x_1 & -x_3-y_5 & 0 & 0 \\
  \end{smallmatrix} \right )\, \text{ and }
  \cA|_{\n'}=\left ( \begin{matrix} 0 & 0 & 0 & 0 \\ -y_1 & 0 & 0 & 0 \\ -y_2 & 0 & 0 & 0 \\ 0 & y_1 & y_2 & 0 \end{matrix} \right )\, .
\end{gather*}
A direct calculation shows that $\cA$ so defined is a skew-symmetric derivation, and that the $GO$ equation $\<\cA T'+[T,T'],T\> = 0$ (see \eqref{eq:golemma}) is satisfied, for all $T' \in \n$. As $\cA$ depends linearly on $T$, one may expect the algebra $(\n,\ip)$ to be even naturally reductive. \hfill $\diamondsuit$

This example also shows that a pseudo-Riemannian $G$-$GO$ nilmanifold with nondegenerate derived algebra loses the property of being 2-step nilpotent already when $\ip'$ is Lorentz (for the case when $\ip'$ is definite, see Remark~\ref{rem:n'def}). 

\section{Proof of Theorem 2: If \texorpdfstring{$ds^2|_{[\n,\n]}$}{ds\unichar{"00B2}|[\unichar{"1D52B},\unichar{"1D52B}]} is degenerate then \\ \texorpdfstring{$\n$}{\unichar{"1D52B}} is a double extension}
\label{s:double}

In this section we consider the case when the restriction of the inner product $\ip$ to the derived algebra $\n'$ is degenerate. 

We start with the following Lemma.
\begin{lemma} \label{l:de}
  Let $(M = G/H, ds^2)$ be a connected pseudo-Riemannian $G$-geodesic orbit nilmanifold where $G =N \rtimes H$, with $N$ nilpotent. Let $\ip$ denote the inner product on $\n$ induced by $ds^2$. Suppose $\ip|_{\n'}$ is degenerate. Let $\m_1$ and $\eg$ be subspaces of $\n$ with the following properties:
     \begin{enumerate}[label=\emph{(\roman*)},ref=\roman*]
      \item \label{it:deincl}
        $\eg \subset \n' \subset \m_1$ \emph{(}so that, in particular, $\m_1$ is an ideal of $\n$\emph{)};

      \item \label{it:dem1z0adh}
        both $\m_1$ and $\eg$ are $\ad_\g(\h)$-invariant;

      \item \label{it:dez0m1}
      $\<\eg, \m_1\>= 0$ and $[\eg,\m_1]=0$;

      \item \label{it:dedim}
      $\dim \m_1 + \dim \eg = \dim \n$.
   \end{enumerate}
   Define the metric nilpotent Lie algebra $\m_0 = \m_1/\eg$ with the inner product $\ip_0$ induced from $\m_1$ \emph{(}this is well-defined by~\eqref{it:dez0m1}\emph{)}, and the pseudo-Riemannian nilmanifold $(M_0 = G_0/H_0, ds_0^2)$, where $G_0 =N_0 \rtimes H_0$, with $N_0$ the \emph{(}simply connected\emph{)} Lie group whose Lie algebra is $\m_0$, $ds_0^2$ is the left-invariant metric on $M_0$ defined by $\ip_0$, and $H_0$ is the maximal connected group of pseudo-orthogonal automorphisms of $\ip_0$.

   Then $(M_0, ds_0^2)$ is a $G_0$-$GO$ pseudo-Riemannian nilmanifold.
\end{lemma}

Note that the fact that the inner product $\ip_0$ as constructed in Lemma~\ref{l:de} is nondegenerate follows from assumption~\eqref{it:dedim}. 

If we denote $m_0=\dim \eg$, then the signature of the metric $ds_0^2$ is $(p-m_0,q-m_0)$, where $(p,q)$ is the signature of $ds^2$. In the settings of Lemma~\ref{l:de}, we say that the metric Lie algebra $(\n,\ip)$ is a \emph{$2m_0$-dimensional double extension of the metric Lie algebra $(\m_0,\ip_0)$}. Informally, to get $(\n,\ip)$, we first take the central extension of $(\m_0,\ip_0)$ by $\eg$ and then the extension of the resulting Lie algebra $\m_1$ by $m_0$-dimensional space of derivations. 

\begin{proof}[Proof of Lemma~\ref{l:de}] 
  To check the $G_0$-$GO$ property for $(\m_0,\ip_0)$ we need to choose (an arbitrary, but fixed) linear complement to $\eg$ in $\m_1$, which, with some abuse of notation, we will still denote $\m_0$. Then $\m_1 = \m_0 \oplus \eg$. For $X, Y \in \m_0$ we have $\<X,Y\>_0=\<X,Y\>$, as $\eg \perp \m_1$ by assumption~\eqref{it:dez0m1}.  We define the Lie bracket $[\cdot,\cdot]_0$ on $\m_0$ by $[X,Y]_0=[X,Y]_{\m_0}$, for $X, Y \in \m_0$. It is easy to see that $(\m_0,\ip_0)$ is isomorphic to the quotient algebra $\m_1/\eg$.

  Let $X \in \m_0$. By the Geodesic Lemma, there exist $A(X) \in \h$ and $k(X) \in \br$ such that for all $Y \in \m_0$ we have $\<[X+A(X),Y],X\> = k(X) \<X,Y\>$. By assumption~\eqref{it:dem1z0adh} we have $[A(X),Y] \in \m_1 \, (= \m_0 \oplus \eg)$, and so we can define an endomorphism $D(X)$ of $\m_0$ by the formula $D(X)Y = [A(X),Y]_{\m_0}$. As $\ad_\g (A(X))$ is skew-symmetric and $\eg \perp \m_1$, the endomorphism $D(X)$ is skew-symmetric relative to $\ip_0$. To see that $D(X)$ is a derivation of the Lie algebra $(\m_0, [\cdot,\cdot]_0)$ we write, for $Y_1, Y_2 \in \m_0$,
  \begin{align*}
    0 & = ([A(X),[Y_1,Y_2]]-[[A(X),Y_1],Y_2]-[Y_1,[A(X),Y_2]])_{\m_0}\\
     & = ([A(X),[Y_1,Y_2]])_{\m_0}-([[A(X),Y_1],Y_2])_{\m_0}-([Y_1,[A(X),Y_2]])_{\m_0}\\
     & = ([A(X),[Y_1,Y_2]_0])_{\m_0}-([D(X)Y_1,Y_2])_{\m_0}-([Y_1,D(X)Y_2])_{\m_0}\\
     & = D(X)([Y_1,Y_2]_0)-[D(X)Y_1,Y_2]_0-[Y_1,D(X)Y_2]_0,
  \end{align*}
  where in the third line, we used the fact that $\eg$ is $\ad_\g(\h)$-invariant, by assumption~\eqref{it:dem1z0adh}.

  It follows that there exists $A_0(X) \in \h_0$, where $\h_0$ is the Lie algebra of $H_0$, such that $D(X)Y = [A_0(X),Y]_0$, for all $X,Y \in \m_0$. Then from assumption~\eqref{it:dez0m1} and the fact that $\<[X+A(X),Y],X\> = k(X) \<X,Y\>$ it follows that $\<[X+A_0(X),Y]_0,X\>_0 = k(X) \<X,Y\>_0$, for all $X,Y \in \m_0$, as required by the Geodesic Lemma.
\end{proof}

  First suppose that $\ip|_{\n'}$ has degeneracy $1$ and is semidefinite. Denote $\vg = (\n')^\perp$ and choose a vector $e$ such that $\n' \cap \vg =\br e$. Denote $\m_1=\n' + \vg$. Note that $e^\perp = \m_1$ and that all four subspaces $\br e, \n', \vg$ and $\m_1$ are $\ad_\g(\h)$-invariant, by Remark~\ref{rem:adhinv}. To be able to apply Lemma~\ref{l:de} (with $\eg=\br e$) we only need to check that $[e, \m_1]=0$. Taking $T'=e$ and $T = X+Y \in \m_1$ in~\eqref{eq:golemma}, where $X \in \n'$ and $Y \in \vg$, we obtain $\<[e,X+Y],X\>=0$, and so $\<[e,X],X\>=\<[e,Y],X\>=0$, for all $X \in \n', \, Y \in \vg$. From the second equation it follows that $[e,Y]$ is a multiple of $e$, and hence $[e,Y]=0$, for all $Y \in \vg$, as $\ad_{\n}(Y)$ is nilpotent. From the first equation we also obtain that $[e,X]$ is a multiple of $e$, as $\ad_{\n'}e$ is nilpotent and skew-symmetric and $\ip|_{\n'}$ has degeneracy $1$ and is semidefinite. Then $[e,X]=0$, for all $X \in \n'$, as $\ad_{\n}(X)$ is nilpotent. Thus $[e,\m_1]=0$, and the claim follows from Lemma~\ref{l:de}, with $\eg=\br e$.

  Next suppose that $\ip|_{\n'}$ has degeneracy $2$ and is semidefinite. Denote $\vg = (\n')^\perp, \; \og=\n' \cap \vg$ and $\s=\n' + \vg$. The restriction of $\ip$ to $\s$ has degeneracy $2$ and is semidefinite. We have $\dim \og = \codim \s = 2$ and $\og^\perp = \s$. Moreover, all four subspaces $\og, \n', \vg$ and $\s$ are $\ad_\g(\h)$-invariant, by Remark~\ref{rem:adhinv}. If $[\og, \s]=0$, we can directly apply Lemma~\ref{l:de} with $\m_1=\s$ and $\eg = \og$, and the claim follows. We therefore assume that $[\og, \s] \ne 0$. Taking $T'=e \in \og$ and $T \in \s$ in~\eqref{eq:golemma} we obtain $\<[e,T],T\>=0$. As the restriction of the inner product to $\s$ is semidefinite, of degeneracy $2$ (and $\s^\perp = \og$), and $\ad_{\s}e$ is both skew-symmetric and nilpotent, we obtain $[e,T] \subset \og$, for all $e \in \og$ and $T \in \s$, and hence $[\s,\og] \subset \og$. We obtain a nilpotent representation of the (nilpotent) algebra $\s$ on the $2$-dimensional space $\og$. By Engel's Theorem, we can find a basis $\{e_1,e_2\}$ for $\og$ such that $[\s, e_2]=0$ and $[T,e_1]=\la(T)e_2$, for all $T \in \s$, where $\la \in \s^*$. As we have assumed that $[\og, \s] \ne 0$, the $1$-form $\la$ is nonzero (but note that $\lambda(\og)=0$). Then $[\og, \s] = \br e_2$, and so by Remark~\ref{rem:adhinv}, the subspace $\br e_2$ is $\ad_\g(\h)$-invariant. Choose two vectors $f_1, f_2 \in \n$ such that $\Span(f_1,f_2) \oplus \s = \n$, and $\<f_i,f_j\>=0,\; \<f_i,e_j\>=\K_{ij}$, for $i,j=1,2$. We claim that the assumptions of Lemma~\ref{l:de} are satisfied with $\eg= \br e_2$ and $\m_1 = \br f_1 \oplus \s$. Indeed, assumptions~\eqref{it:deincl} and \eqref{it:dedim} are obviously true, and for assumption~\eqref{it:dem1z0adh} we note that $\m_1=(\br e_2)^\perp$ by construction, and hence $\m_1$ is $\ad_\g(\h)$-invariant by Remark~\ref{rem:adhinv}, as $\br e_2$ is. It remains to show that $[\m_1,e_2]=0$. As we already know that $[\s, e_2]=0$, it suffices to show that $[f_1,e_2]=0$. Taking $T'=e_2$ and $T =\xi f_1+X\in \m_1$, where $X \in \s, \xi \in \br$, in~\eqref{eq:golemma} (and using the fact that $\br e_2$ and $\m_1$ are orthogonal, $\ad_\g(\h)$-invariant subspaces) we obtain $\<[f_1,e_2],\xi f_1+X\>=0$, for all $X \in \s, \xi \in \br$. It follows that $[f_1,e_2]$ is a multiple of $e_2$, which must be zero, by nilpotency. The claim now follows from Lemma~\ref{l:de}.

  The last case to consider is the one when $\ip|_{\n'}$ has degeneracy $1$ and index $1$. This is the most involved case. As above, we denote $\vg = (\n')^\perp$ and choose a vector $e$ such that $\n' \cap \vg =\br e$. Denote $\m_1=\n' + \vg$, so that $e^\perp = \m_1$. The subspaces $\br e, \n', \vg$ and $\m_1$ are $\ad_\g(\h)$-invariant, by Remark~\ref{rem:adhinv}. We claim that the assumptions of Lemma~\ref{l:de} are satisfied with $\eg=\br e$. It is easy to see that the only fact we need to establish is that $[e, \m_1]=0$.

  The proof is completed by the following proposition.

  {
  \begin{proposition} \label{p:em10}
  In the above notation, the vector $e$ lies in the centre of $\m_1$.
\end{proposition}

\begin{proof}
Denote $m = \dim \n'$. Seeking a contradiction we assume that $[\m_1,e] \ne 0$.

Let $\f=\{T \in \n \, : \, \<[T,X],X\>=0, \text{ for all } X \in \n'\}$. It is easy to see that $\f$ is a subalgebra of $\n$.
{
\begin{lemma} \label{l:skewonn'}
In the above notation, the following holds.
   \begin{enumerate}[label=\emph{(\alph*)},ref=\alph*]
      \item \label{it:vgskewonn'}
        $\vg \subset \f$ and $[\f,e]=0$ \emph{(}so, in particular, $[e,\vg]=0$\emph{)}.
	
	  \item \label{it:basisforn'} 
        There exists a hyperplane $\n_0 \subset \n'$, with $\n_0 \oplus \br e = \n'$, and a basis $\{e_1, \dots, e_{m-1}\}$ for $\n_0$ such that relative to the basis $\{e_1, \dots, e_{m-1},e\}$ for $\n'$, we have
\begin{equation} \label{eq:ip'ugii}
    \ip|_{\n'} = \left ( \begin{smallmatrix} 0 & 0 & 1 & 0\\ 0 & I_{m-3} & 0 & 0 \\  1 & 0 & 0 & 0 \\  0 & 0 & 0 & 0 \end{smallmatrix} \right ), \,
    \ad_{\n'}T = \left ( \begin{smallmatrix} 0 & 0 & 0 & 0 \\ VT & 0_{m-3} & 0 & 0 \\ 0 & -(VT)^t & 0 & 0 \\ a(T) & (WT)^t & 0 & 0 \end{smallmatrix}\right )
    \text{ \emph{and} }
	\ad_{\n'}e = \left ( \begin{smallmatrix} 0 & 0 & 0 & 0 \\ u & 0_{m-3} & 0 & 0 \\ 0 & -u^t & 0 & 0 \\ 0 & 0 & 0 & 0 \end{smallmatrix}\right ),
\end{equation}
        for all $T \in \f$, where $u \in L:=\Span(e_2, \dots, e_{m-2}),\, u \ne 0, \, a \in \f^*$ and $V,W: \f \to L$ are linear maps. In particular, $[\f,e_{m-1}]=0$.

      \item \label{it:Ae}
      The subspaces $\Span(u, e_{m-1})$ and $\br e_{m-1}$ are $\ad_\g(\h)$-invariant. Moreover, for any $A \in \h$, we have $[A,e]=\a(A)e, \, [A,u] = \b(A) e_{m-1}, \, [A,e_{m-1}] = \gamma(A)e_{m-1}$, for some $\a,\b,\gamma \in \h^*$.
   \end{enumerate}
\end{lemma}
\begin{proof}
  For assertion~\eqref{it:vgskewonn'}, the fact that $[\f,e]=0$ easily follows: for all $T \in \f$ and $X \in \n'$, we have $0=\<[T,X],e\>=-\<[T,e],X\>$. Therefore $[T,e]$ is a multiple of $e$, which must be zero as $\ad_{\n'}T$ is nilpotent.

  To see that $\vg \subset \f$, take $T'=Y \in \vg$ and $T=X \in \n'$ in~\eqref{eq:golemma}. As $\vg$ and $\n'$ are orthogonal, $\ad_\g(\h)$-invariant subspaces, se obtain $\<[X,Y],X\>=0$, as required.

  For assertion~\eqref{it:basisforn'}, we note that the subspace $[e,\n'] \subset \n'$ does not contain $\br e$ (indeed, for no $X \in \n'$ we can have $[X,e]=e$, as $\ad_{\n'}X$ is nilpotent). Choose a linear complement $\n_0$ to $\br e$ in $\n'$ in such a way that $\n_0 \supset [e,\n']$. The restriction of $\ip$ to $\n_0$ is nondegenerate and is of Lorentz signature. For $T \in \f$ we define the endomorphism $\phi_T$ of $\n_0$ by $[T,X]=\phi_T X + \mu(X)e$, for $X \in \n_0$. For every $T \in \f$, the endomorphism $\phi_T$ is skew-symmetric (as $\<e,\n'\>=0$) and nilpotent (as $[\f,e] = 0$ by assertion~\eqref{it:vgskewonn'}). Moreover, the map $\phi: \f \to \so(\n_0,\ip|_{\n_0})$ sending $T$ to $\phi_T$ is a Lie algebra homomorphism (as $[\f,e] = 0$). Considering the Iwasawa decomposition of the Lie algebra $\so(m -2, 1)=\so(\n_0,\ip|_{\n_0})$, by the argument similar to that in the proof of Proposition~\ref{p:dernondeg2} (see equations~\eqref{eq:ip'ug} and~\eqref{eq:advg}) we can construct a basis $\{e_1, \dots, e_{m-1}\}$ for $\n_0$ such that the restriction of $\ip$ to $\n'$ relative to the basis $\{e_1, \dots, e_{m-1},e\}$ for $\n_0$ has the form as given in~\eqref{eq:ip'ugii}, and moreover, there is a linear map $V: \f \to L \,(=\Span(e_2, \dots, e_{m-2})$ such that for all $T \in \f$ we have $\phi_T e_1 = VT, \; \phi_T e_{m-1}=0$ and $\phi_T e_i = -\<VT,e_i\>e_{m-1}$, for $i=2, \dots, m-2$.

  It follows that for some linear map $W: \f \to L$ and linear forms $a, b \in \f^*$, we have
  \begin{equation}\label{eq:adfonn'}
  \begin{gathered}
  [T, e_1] = VT + a(T)e, \qquad [T, e_{m-1}]=b(T)e, \\ [T, e_i] = -\<VT,e_i\>e_{m-1} + \<WT,e_i\>e, \quad i=2, \dots, e_{m-2},
  \end{gathered}
  \end{equation}
  for all $T \in \f$. In particular, taking $T=e$ and using the fact that $[e,\n_0] \subset \n_0$ (by construction of $\n_0$) we obtain $We=0$ and $a(e)=b(e)=0$. Thus $\ad_{\n'}e$ has the form as given in~\eqref{eq:ip'ugii}, where we denote $u=Ve \in L$. Then from~\eqref{eq:adfonn'} we obtain $[e,u]=-\|u\|^2e_{m-1}$, and so $[T,[e,u]]=-\|u\|^2 b(T) e$, for all $T \in \f$. But $[T,e]=0$ by assertion~\eqref{it:vgskewonn'}, which gives $[[T,u],e]= \|u\|^2 b(T) e$. As $\ad_{\n'}[T,u]$ is nilpotent we get $\|u\|^2 b(T) = 0$, for all $T \in \f$. If $b \ne 0$ we get $u=0$ (as the restriction of $\ip$ to $L$ is definite), and so $[e,\n']=0$. As $[e,\vg]=0$ by assertion~\eqref{it:vgskewonn'} we obtain $[e,\m_1]=0$ contradicting our assumption. Therefore $b=0$, and then equations~\eqref{eq:adfonn'} imply that $\ad_{\n'}T$ has the form given in~\eqref{eq:ip'ugii}.

  The last statement in assertion~\eqref{it:basisforn'} follows from~\eqref{eq:ip'ugii}.

  For assertion~\eqref{it:Ae}, we note that from~\eqref{eq:ip'ugii} we obtain $[e,\n']=\Span(u, e_{m-1})$ and $[e,[e,\n']]=\br e_{m-1}$, and so the first claim follows from Remark~\ref{rem:adhinv}. Then the second claim also follows (note that the $u$-component of $[A,u]$ vanishes as $\<[A,u],u\>=0$ and $u$ lies in the subspace $L$ with a definite inner product).
\end{proof}
}
Let now $f \notin \m_1$ be a null vector such that $f \perp \n_0$ and $\<f,e\>=1$ (the choice of such an $f$ is not unique); note that $\m_1 \oplus \br f = \n$.

{
\begin{lemma} \label{l:fe}
    The following holds:  
    \begin{align}\label{eq:hfe}
      & [\h,[f,e]]=0, \\
      & [e,[f,\vg]]=[[f,e],\vg]=0, \label{eq:fev} \\
      & [\vg, [f,[f,e]]] = 0, \label{eq:ffev} \\
      & [f,[f,[f,e]]] = 0. \label{eq:fffe}
    \end{align}
\end{lemma}
\begin{proof}
  From Lemma~\ref{l:skewonn'}\eqref{it:Ae}, for any $A \in \h$, we have $\<[A,f],e\> = -\<[A,e],f\>= -\a(A)$, $\<[A,f],e_{m-1}\> = -\<[A,e_{m-1}],f\>= 0$ and $\<[A,f],u\> = -\<[A,u],f\>= 0$, as $e_{m-1} \in \n_0 \subset f^\perp$. It follows that $[A,f]=-\a(A)f + Y + X$, where $Y \in \vg$ and $X \in (\Span(e_{m-1},u))^\perp \cap \n'$. Then $[A,[f,e]]=[[A,f],e]+[f,[A,e]]=[-\a(A)f + Y + X,e]+[f, \a(A)e]=0$, as $[Y,e]=0$ by Lemma~\ref{l:skewonn'}\eqref{it:vgskewonn'} and $[X,e]=0$ by~\eqref{eq:ip'ugii}. This proves~\eqref{eq:hfe}.

  Take in~\eqref{eq:golemma} $T'=[f,e]$ and a non-null vector $T=\mu f + Y + X$, where $Y \in \vg, \, X \in \n'$ and $\mu \in \br$. As $[A,[f,e]]=0$ by~\eqref{eq:hfe} and $\vg \perp \n'$ we obtain $\<[\mu f + Y + X, [f,e]], \mu f + X\>=0$, for all $Y \in \vg, \, X \in \n'$ and $\mu \in \br$, by continuity, from which we get
  \begin{equation} \label{eq:yfex}
  \begin{gathered}
    \<[Y, [f,e]], X\>=0, \quad \<[X, [f,e]], X\>=0, \quad \<[f, [f,e]], f\>=0, \\ \<[f, [f,e]], X\>+\<[X, [f,e]], f\>=0,
  \end{gathered}
  \end{equation}
  for all $Y \in \vg, \, X \in \n'$.

  The first equation of~\eqref{eq:yfex} implies that $[Y, [f,e]]$ is a multiple of $e$. But $[Y, [f,e]]=[[Y, f],e]$, as $[Y,e]=0$ by Lemma~\ref{l:skewonn'}\eqref{it:vgskewonn'}, and so $[Y, [f,e]]=[[Y, f],e]=0$, as $\ad_{\n'}[Y,f]$ is nilpotent. This proves~\eqref{eq:fev}.

  We now consider the last equation of~\eqref{eq:yfex}. We have $\<[X, [f,e]], f\>=\<-[e, [X,f]] - [f,[e,X]], f\>$. As $[e,\n'] \subset \n_0$ (by construction of $\n_0$) and $f \perp \n_0$, we have $\<[e, [X,f]], f\>=0$. Taking $X=x_1e_1+\tilde{x}+x_{m-1}e_{m-1}+xe \in \n$, where $x_1, x_{m-1},x \in \br$ and $\tilde{x} \in L\,(=\Span(e_2,\dots,e_{m-2}))$ we obtain from~\eqref{eq:ip'ugii} that $[e,X] = x_1 u +\<\tilde{x},u\>e_{m-1}$, which gives $\<[X, [f,e]], f\>= -\< [f,[e,X]], f\> = -\<[f,x_1 u +\<\tilde{x},u\>e_{m-1}],f\>= -x_1 \<[f,u],f\> - \<\tilde{x},u\> \<[f,e_{m-1}],f\>$. From the last equation of~\eqref{eq:yfex} we obtain $[f,[f,e]]=\<[f,u],f\>e_{m-1} + \<[f,e_{m-1}],f\>u + \eta e$, for some $\eta \in \br$. But then $\eta = 0$ from the third equation of~\eqref{eq:yfex}. Moreover, as $e \in \vg$, we get $[e,[f,e]]=0$ by~\eqref{eq:fev}, which implies $[e,[f,[f,e]]]=0$. Substituting the above expression for $[f,[f,e]]$ and using~\eqref{eq:ip'ugii} we find
  \begin{equation}\label{eq:fem-1f}
  \<[f,e_{m-1}],f\>=0, \text{ and so } [f,[f,e]]=\<[f,u],f\>e_{m-1}.
  \end{equation}
  But from Lemma~\ref{l:skewonn'}\eqref{it:vgskewonn'},~\eqref{it:basisforn'} we have $\vg \subset \f$ and $[\f,e_{m-1}]=0$ which implies $[[f,[f,e]],\vg]=0$, by the second equation of~\eqref{eq:fem-1f} . This establishes~\eqref{eq:ffev}.

  From the second equation of~\eqref{eq:yfex} we obtain that $[f,e] \in \f$, and so $\ad_{\n'}[f,e]$ has the form given in~\eqref{eq:ip'ugii}, in particular, $[[f,e],u] \in \Span(e_{m-1},e)$. As $[e,[f,u]] \in [e,\n'] = \Span(u, e_{m-1})$ (by~\eqref{eq:ip'ugii}), we obtain $[f,[e,u]] = [[f,e],u] + [e,[f,u]] \in \Span(u,e_{m-1},e)$. But $[e,u]=-\|u\|^2e_{m-1}$ by~\eqref{eq:ip'ugii}, and so we obtain $[f,e_{m-1}]= \rho_1 u +\rho_2 e_{m-1} + \rho_3 e$, for some $\rho_1, \rho_2, \rho_3 \in \br$. Then from the first equation of~\eqref{eq:fem-1f} we get $\rho_3=0$. Moreover, as $[f,e] \in \f$, from~\eqref{eq:ip'ugii} we find $[[f,e],e_{m-1}]=0$ which implies $[e,[f,e_{m-1}]]=0$ (since $[e,e_{m-1}]=0$ by~\eqref{eq:ip'ugii}). From the expression for $[f,e_{m-1}]$ above we obtain $[e,\rho_1 u +\rho_2 e_{m-1}]=0$ which implies $\rho_1=0$, again by~\eqref{eq:ip'ugii}. Therefore $[f,e_{m-1}]= \rho_2 e_{m-1}$ which gives $[f,e_{m-1}]=0$, by nilpotency. Now equation~\eqref{eq:fffe} follows from the second equation of~\eqref{eq:fem-1f}.
\end{proof}
}

We can now complete the proof of the proposition. We have $\n = \br f \oplus \m_1 = \br f \oplus (\vg + \n') = (\br f \oplus \vg) + \n'$. It follows that the subspace $\Vg=\br f \oplus \vg$ contains some linear complement to $\n'$ in $\n$, and hence generates $\n$. Then $\n = \Vg + [\Vg,\Vg]+[\Vg,[\Vg,\Vg]]+ \dots$, and so $\n'=[\Vg,\Vg]+[\Vg,[\Vg,\Vg]]+ \dots$. As we already know that $[e,\vg]=0$ (by Lemma~\ref{l:skewonn'}\eqref{it:vgskewonn'}), to show that $[e,\m_1]=0$ it suffices to prove that $[e,\n']=0$, that is, to prove that $[e,[T_1,[T_2,[\dots,[T_{r-1},T_r]\dots]]]]=0$, where $r \ge 2$, and where, for every $i=1, \dots, r$, we have either $T_i=f$ or $T_i \in \vg$. The proof goes by induction by $r \ge 2$. If $r=2$ the claim follows from the facts that $[e,\vg]=0$ and that $[e,[f,\vg]]=0$ (by~\eqref{eq:fev}). Suppose $r > 2$. If $T_1 \in \vg$, then the claim follows by the induction assumption from the fact that $[e,T_1]=0$. Suppose $T_1=f$. Then by the induction assumption it suffices to prove that $[[e,f],[T_2,[T_3,[\dots,[T_{r-1},T_r]\dots]]]]=0$. If $T_2 \in \vg$, the claim follows from the fact that $[e,[f,\vg]]=0$ (by~\eqref{eq:fev}) and the induction assumption (or from~\eqref{eq:ffev} if $r=3$). Suppose $T_2=f$. Then it suffices to prove that $[[[e,f],f],[T_3,[\dots,[T_{r-1},T_r]\dots]]=0$. But $[[[e,f],f],f]=0$ by~\eqref{eq:fffe} and $[[[e,f],f],\vg]=0$ by~\eqref{eq:ffev}. It follows that $[[[e,f],f],T_i]=0$, for all $i=3, \dots, r$, which completes the proof of Proposition~\ref{p:em10}.
\end{proof}
  }
With Proposition~\ref{p:em10}, application of Lemma~\ref{l:de} completes the proof of Theorem \ref{th:deg}.

\end{document}